\documentclass[11pt]{article}

\usepackage{etex}
\usepackage{graphicx}
\usepackage{amsmath,amsthm}
\usepackage{amsfonts}
\usepackage{amssymb}
\usepackage{hyperref}
\usepackage{array}
\usepackage{multirow}
\usepackage{longtable}
\usepackage{subfigure}
\usepackage{setspace}
\usepackage{pgf,tikz}
\usepackage[rflt]{floatflt}

\usetikzlibrary{shapes.geometric}

\usepackage{mathrsfs}
\usetikzlibrary{arrows}

\usetikzlibrary{decorations.markings}
\tikzset{->-/.style={decoration={
  markings,
  mark=at position #1 with {\arrow{>}}},postaction={decorate}}}

\usetikzlibrary{calc}
\usetikzlibrary{graphs}
\usetikzlibrary{positioning}

\newtheorem{theorem}{Theorem}[section]

\newtheorem{corollary}[theorem]{Corollary}
\newtheorem{lemma}[theorem]{Lemma}
\newtheorem{conjecture}[theorem]{Conjecture}

\newtheorem{fact}[theorem]{Fact}
\newtheorem{claim}[theorem]{Claim}

\theoremstyle{definition}

\newtheorem{definition}[theorem]{Definition}

\newcommand{\blackvertex}[1]{\fill (#1) circle (2pt);}
\newcommand{\whitevertex}[1]{\fill (#1) circle (2pt);\fill[white] (#1) circle (1.5pt);}

\usepackage{color}
\definecolor{blue}{rgb}{0,0,1}

\newcolumntype{I}{>{\centering\arraybackslash} m{.15\linewidth}} 
\newcolumntype{A}{>{\centering\arraybackslash} m{.18\linewidth}} 
\newcolumntype{V}{>{\centering\arraybackslash} m{.22\linewidth}} 
\newcolumntype{W}{>{\centering\arraybackslash} m{.28\linewidth}}

\usepackage{fullpage}

 \newenvironment{cit}
{
    \begin{list}{- \ }{}
        \setlength{\topsep}{0pt}
        \setlength{\parskip}{0pt}
        \setlength{\partopsep}{0pt}
        \setlength{\parsep}{0pt}         
        \setlength{\itemsep}{0pt} 
}
{
    \end{list} 
}

\newcounter{confignum}	
\newcommand{\configuration}[1]{\refstepcounter{confignum}\label{#1}\ref{#1}}

\def\ex{\operatorname{ex}}

\newcounter{casenum}	
\newcounter{subcasenum}	
\numberwithin{subcasenum}{casenum}
\newcounter{subsubcasenum}	
\numberwithin{subsubcasenum}{subcasenum}

\renewcommand{\thecasenum}{\arabic{casenum}}

\newcommand{\hs}{\hspace{1mm}}

\newcounter{stagenum}

\newenvironment{mycases}
{
  \list{}{%
    \leftmargin0.5cm   % this is the adjusting screw
    \rightmargin0cm%\leftmargin
  }
  \item\relax
	\setcounter{casenum}{0}
}
{	
	\endlist
}

\newcommand{\mycase}[1]{
	\vspace{0.5em}
	
	\refstepcounter{casenum}
	\noindent\hspace{-0.5cm}\textit{Case \thecasenum: #1}
}

\def\ex{\operatorname{ex}}

\begin{document}
%\linenumbers

%\spacing{2}
\title{$(4,2)$-choosability of planar graphs with forbidden structures}
\author{
Zhanar Berikkyzy$^1$
\and
Christopher Cox$^2$
\and
Michael Dairyko$^1$
\and
Kirsten Hogenson$^1$
\and
Mohit Kumbhat$^1$
\and
Bernard Lidick\'y$^{1,3}$
\and
Kacy Messerschmidt$^1$
\and
Kevin Moss$^1$
\and
Kathleen Nowak$^1$
\and
Kevin F. Palmowski$^1$
\and
Derrick Stolee$^{1,4}$
}

\maketitle

\begin{abstract}
All planar graphs are 4-colorable and 5-choosable, while some planar graphs are not 4-choosable.
Determining which properties guarantee that a planar graph can be colored using lists of size four has received significant attention.
In terms of constraining the structure of the graph, for any $\ell \in \{3,4,5,6,7\}$, a planar graph is 4-choosable if it is $\ell$-cycle-free.
In terms of constraining the list assignment, one refinement of $k$-choosability is \emph{choosability with separation}.
A graph is \emph{$(k,s)$-choosable} if the graph is colorable from lists of size $k$ where adjacent vertices have at most $s$ common colors in their lists.
Every planar graph is $(4,1)$-choosable, but there exist planar graphs that are not $(4,3)$-choosable.
It is an open question whether planar graphs are always $(4,2)$-choosable.
A \emph{chorded $\ell$-cycle} is an $\ell$-cycle with one additional edge.
We demonstrate for each $\ell \in \{5,6,7\}$ that a planar graph is $(4,2)$-choosable if it does not contain chorded $\ell$-cycles.
\end{abstract}

\footnotetext[1]{Department of Mathematics, Iowa State University, Ames, IA, U.S.A. \texttt{$\{$zhanarb,mdairyko,kahogens,mkumbhat,lidicky,kacymess,kmoss,knowak,kpalmow,dstolee$\}$@iastate.edu}}
\footnotetext[2]{Department of Mathematical Sciences, Carnegie Mellon University, Pittsburg, PA, U.S.A. \texttt{cocox@andrew.cmu.edu}}
\footnotetext[3]{Supported by NSF grant DMS-1266016.}
\footnotetext[4]{Department of Computer Science, Iowa State University, Ames, IA, U.S.A.}

\section{Introduction}

A \emph{proper coloring} is an assignment of colors to the vertices of a graph $G$ such that adjacent vertices are assigned distinct colors. 
A \emph{$(k,s)$-list assignment $L$} is a function that assigns a list $L(v)$ of $k$ colors to each vertex $v$ so that $|L(v)\cap L(u)|\le s$ whenever $uv\in E(G)$. 
A proper coloring $\phi$ of $G$ such that $\phi(v) \in L(v)$ for all $v\in V(G)$ is called an \emph{$L$-coloring}.
We say that a graph $G$ is \emph{$(k,s)$-choosable} if, for any $(k,s)$-list assignment $L$, there exists an $L$-coloring of $G$.
We call this variation of graph coloring \emph{choosability with separation.} 
Note that when a graph is $(k,k)$-choosable, we simply say it is \emph{$k$-choosable}.
Observe that if $G$ is $(k,t)$-choosable, then $G$ is $(k,s)$-choosable for all $s \le t.$  
A notable result from Thomassen \cite{thomassen} states that every planar graph is 5-choosable, so it follows that all planar graphs are $(5,s)$-choosable for all $s\le5.$

Forbidding certain structures within a planar graph is a common restriction used in graph coloring. 
Theorem~\ref{thm:31choosableconditions} summarizes the current knowledge on $(3,1)$-choosability of planar graphs.
\v{S}krekovski \cite{skrekovski} conjectured that all planar graphs are $(3,1)$-choosable; this question is still open and is presented below as Conjecture \ref{conj:31choosable}.

\begin{conjecture}[\v{S}krekovski \cite{skrekovski}]\label{conj:31choosable}
If $G$ is a planar graph, then $G$ is $(3,1)$-choosable.
\end{conjecture}

\begin{theorem}\label{thm:31choosableconditions}
A planar graph $G$ is $(3,1)$-choosable if $G$ avoids any of the following structures:
\begin{cit}
\item 3-cycles (Kratochv\'{i}l, Tuza, Voigt \cite{krat}). % Choi, Lidick\'y, Stolee \cite{choi}).
\item 4-cycles (Choi, Lidick\'y, Stolee \cite{choi}).
\item 5-cycles and 6-cycles (Choi, Lidick\'y, Stolee \cite{choi}).
\end{cit}
\end{theorem}

In this paper, we focus on 4-choosability with separation.  
Kratochv\'{i}l, Tuza, and Voigt~\cite{krat} proved that all planar graphs are $(4,1)$-choosable, while Voigt \cite{voigt1993list} demonstrated that there exist planar graphs that are not $(4,3)$-choosable. 
It is not known if all planar graphs are $(4,2)$-choosable.  
%This question is still open, and is presented below as Conjecture \ref{conj:42choosable}.

\begin{conjecture}[Kratochv\'{i}l, \emph{et al.} \cite{krat}]\label{conj:42choosable}
If $G$ is a planar graph, then $G$ is $(4,2)$-choosable.
\end{conjecture}

\begin{theorem}[Kratochv\'{i}l, \emph{et al.} \cite{krat}]\label{thm:41choosable}
If $G$ is a planar graph, then $G$ is $(4,1)$-choosable.
\end{theorem}

Theorem~\ref{thm:41choosable} was strengthened by Kierstead and Lidick\'{y}  \cite{kl}, where it is shown that we can allow an independent set of vertices to have lists of size 3 rather than 4.  

\begin{theorem}[Kierstead and Lidick\'{y} \cite{kl}]\label{thm:41choosablestrengthen}
Let $G$ be a planar graph and $I\subseteq V(G)$ be an independent set.  If $L$ assigns lists of colors to $V(G)$ such that $|L(v)|\geq 3$ for every $v\in I$, and $|L(v)|=4$ for every $v\in V(G)\setminus I$, and $|L(u) \cap L(v)| \leq 1$ for all $uv \in E(G)$, then $G$ has an $L$-coloring.
\end{theorem}

In addition to the work summarized above, there are several results regarding 4-choosability. % with forbidden structures.  
A graph is \emph{$k$-degenerate} if each of its subgraphs has a vertex of degree at most $k$.
Euler's formula implies a planar graph with no 3-cycles is 3-degenerate and hence 4-choosable.
This and other similar results are listed below in Theorem \ref{thm:4choosableconditions}.  
For the last result in Theorem \ref{thm:4choosableconditions}, note that a \emph{chorded $\ell$-cycle} is an $\ell$-cycle with an additional edge connecting two of its non-consecutive vertices.

\begin{theorem}\label{thm:4choosableconditions}
A planar graph $G$ is 4-choosable if $G$ avoids any of the following structures:
\begin{cit}
\item 3-cycles (folklore).
\item 4-cycles (Lam, Xu, Liu, \cite{lam}).
\item 5-cycles (Wang and Lih \cite{wang}).
\item 6-cycles (Fijavz, Juvan, Mohar, and \v{S}krekovski~\cite{fijavz}).
\item 7-cycles (Farzad \cite{farzad}).
\item Chorded 4-cycles and chorded 5-cycles (Borodin and Ivanova \cite{borodin}).
\end{cit}
\end{theorem}

Our main results in this paper are listed below in Theorem \ref{thm:42choosableconditions}.  
Note that a \emph{doubly-chorded $\ell$-cycle} is a chorded $\ell$-cycle with an additional edge.
%We build on the idea of restricting chorded cycles in order to address the open problem of $(4,2)$-choosability of planar graphs.

\begin{theorem}\label{thm:42choosableconditions}
A planar graph $G$ is $(4,2)$-choosable if $G$ avoids any of the following structures:
\begin{cit}
\item Chorded 5-cycles.
\item Chorded 6-cycles.
\item Chorded 7-cycles.
\item Doubly-chorded 6-cycles and doubly-chorded 7-cycles.
\end{cit}
\end{theorem}

We prove each case of Theorem~\ref{thm:42choosableconditions} separately.
In Section~\ref{sec:5cycles}, we forbid chorded 5-cycles (see Theorem~\ref{thm:5cycles}).
In Section~\ref{sec:6cycles}, we forbid chorded 6-cycles (see Theorem~\ref{thm:cc6}); we use parts of this proof to also prove the case when forbidding doubly-chorded 6-cycles and doubly-chorded 7-cycles (see Corollary~\ref{cor:dcc67}).
In Section~\ref{sec:7cycles}, we forbid chorded 7-cycles (see Theorem~\ref{thm:cc7strengthened}). 
There are many features common to all of these proofs, which we detail in Sections~\ref{sec:overview} and~\ref{sec:reducible}.

\subsection{Preliminaries and Notation}

Refer to \cite{west} for standard graph theory terminology and notation. 
Let $G$ be a graph with a vertex set $V(G)$ and an edge set $E(G)$; let $n(G) = |V(G)|$.
We use $K_n$, $C_n$, and $P_n$ to denote the complete graph, cycle graph, and path graph, respectively, each on $n$ vertices.
The \emph{open neighborhood} of a vertex, denoted $N(v)$, is the set of vertices adjacent to $v$ in $G$; the \emph{closed neighborhood}, denoted $N[v]$, is the set $N(v) \cup \{v\}$.
The \emph{degree} of a vertex $v$, denoted $d_G(v)$, is the number of vertices adjacent to $v$ in $G$; we write $d(v)$ when the graph $G$ is clear from the context. 
If the degree of a vertex $v$ is $k$, we call $v$ a \emph{$k$-vertex}; if the degree of $v$ is at least $k$, we call $v$ a \emph{$k^{+}$-vertex}.  
The \emph{length} of a face $f$, denoted $\ell(f)$, is the length of the face boundary walk.
If the length of a face $f$ is $k$, we call $f$ a \emph{$k$-face}; if the length of $f$ is at least $k$, we call $f$ a \emph{$k^{+}$-face}. 
%A \emph{cluster} is a subgraph of $G$ which consists of a minimal set of $3$-faces such that no other $3$-face is adjacent to a $3$-face in the set \cite{farzad}. We have listed all possible clusters in a planar graph without $6$-cycles in Figure~\ref{fig:clusters}.  

%\clearpage
\section{Overview of Method}\label{sec:overview}

All of our main results use the discharging method.
We refer the reader to the surveys by Borodin~\cite{BorodinSurvey} and Cranston and West~\cite{CW} for an introduction to discharging, which is a method commonly used to obtain results on planar graphs. % that forbid certain structures.
For real numbers $a_v, a_f, b$, we define initial charge values $\mu(v) = a_vd(v)-b$ for every vertex $v$ and $\nu(f) = a_f\ell(f)-b$ for every face $f$. 
If $a_v > 0$, $a_f >0$ and $2a_v + 2a_f = b > 0$, then Euler's formula implies that $\sum_v \mu(v) + \sum_f \nu(f) = -2b$, and the total charge on the entire graph is negative.
We then define \emph{discharging rules} that describe a method for moving charge value among vertices and faces while conserving the total charge value.
We demonstrate that if $G$ is a ``minimal counterexample'' to our theorem, then every vertex and face ends with nonnegative charge after the discharging process, which is a contradiction.
Intuitively, this process works well when forbidding a structure (such as a short chorded cycle) with low charge.

In Section~\ref{sec:reducible}, we concretely define \emph{reducible configurations}.
Loosely, a reducible configuration is a structure $C$ in a graph $G$ with $(4,2)$-list assignment $L$ where any $L$-coloring of $G-C$ extends to an $L$-coloring of $G$.
If we are looking for a minimal example of a graph that is not $(4,2)$-choosable, then none of these reducible configurations appear in the graph.
We define a large list of configurations, \ref{cycle4}--\ref{lastconf} (see Figure~\ref{fig:configurations}), and prove they are reducible using various generic constructions.
The configurations \ref{cycle4}--\ref{d2} are used when forbidding chorded 6- or 7-cycles, while the configurations  \ref{d11}--\ref{lastconf} are used when forbidding chorded 5-cycles.
The use of different configurations is due to differences in our discharging arguments.

In Section~\ref{sec:5cycles}, we forbid chorded 5-cycles and every 3-face is adjacent to at most one other 3-face.
Moreover, 3-faces are not adjacent to 4-faces.
Thus, our initial charge function in this case guarantees that the only objects with negative initial charge are 4- and 5-vertices.

In Sections~\ref{sec:6cycles} and \ref{sec:7cycles}, we use a different discharging strategy.
Our initial charge values guarantee that the only objects of negative charge are 3-faces.
Thus, our discharging rules are designed to send charge from $5^+$-faces and $4^+$-vertices to 3-faces.
However, as we forbid chorded 6-cycles or chorded 7-cycles, there may be many 3-faces very close to each other.

\input{figures/clusters.tex}

If $G$ is a plane graph and $G^*$ is its dual, then let $F_3$ be the set of 3-faces of $G$ and let $G^*_3$ be the induced subgraph of $G^*$ with vertex set $F_3$.
A \emph{cluster} is a maximal set of 3-faces that are connected in $G^*$, i.e., a connected component of $G^*_3$.
Note that two 3-faces sharing an edge are adjacent in $G^*$, and two 3-faces sharing only a vertex are not adjacent in $G^*$.
See Figure~\ref{fig:clusters} for a list of the clusters with maximum cycle length six and every internal vertex of degree at least four.
In these figures, the outer cycle is not necessarily a facial cycle, any area filled with gray is not a face, and a pair of square vertices represent a single vertex.
Additionally, bold edges describe \emph{separating 3-cycles}, which are cycles in a plane graph whose exterior and interior regions both contain vertices not on the cycle.
These figures are based on the list of clusters used by Farzad \cite{farzad} in the proof that 7-cycle-free planar graphs are 4-choosable.

For $k \in \{1,2\}$, there is exactly one way to arrange $k$ 3-faces in a cluster.
A \emph{triangle} is a cluster containing exactly one 3-face; see (K3).
A \emph{diamond} is a cluster containing exactly two 3-faces; see (K4).
For $k \geq 3$, there are multiple ways to arrange $k$ 3-face in a cluster.
A \emph{$k$-fan} is a cluster of $k$ 3-faces all incident to a common vertex of degree at least $k+1$; see (K5a) and (K6b).
A \emph{$k$-wheel} is a cluster of $k$ 3-faces all incident to a common vertex of degree exactly $k$; see (K5b) and (K6e).
Note that the vertex incident to all faces of a $3$-wheel has degree 3.
A \emph{$k$-strip} is a cluster of $k$ 3-faces $f_1,\dots,f_k$ where the boundaries of the 3-faces are disjoint except that $f_i$ and $f_{i+1}$ share an edge for $i \in \{1,\dots,k-1\}$ and $f_i$ and $f_{i+2}$ share a vertex for $i\in \{1,\dots, k-2\}$; see (K5a) and (K6a).

If $f_1,\dots,f_k$ are the 3-faces in a cluster, then we will prove that the total charge on $f_1,\dots,f_k$ after discharging is nonnegative.
Thus, some of the 3-faces may have negative charge, but this is balanced by other 3-faces in the cluster having positive charge.
Hence, our proofs end with a list of all possible cluster types and verifying that each has nonnegative total charge.

While there are 23 total clusters that avoid chorded 7-cycles, we do not have that many cases to check.
The clusters (K5c) and (K6g)--(K6r) have three bold edges, demonstrating a separating 3-cycle.
We avoid checking these cases by using a strengthened coloring statement (see Theorem~\ref{thm:cc7strengthened}) that allows our minimal counterexample to not contain any separating 3-cycles.

%%%%%%%%%%%%%%%%%%%%%%%%%%%%%%%%%%%%%%%%%%%%%%%%%%%%%%%%%%%%
%%%%%%%%%%%%%%%%%%%%%%%%%%%%%%%%%%%%%%%%%%%%%%%%%%%%%%%%%%%%
%%%%%%%%%%%%%%%%%%%%%%%%%%%%%%%%%%%%%%%%%%%%%%%%%%%%%%%%%%%%
%%%%%%%%%%%%%%% Reducible Configurations %%%%%%%%%%%%%%%%%%%%%%%%%%%%%%%%
%%%%%%%%%%%%%%%%%%%%%%%%%%%%%%%%%%%%%%%%%%%%%%%%%%%%%%%%%%%%
%%%%%%%%%%%%%%%%%%%%%%%%%%%%%%%%%%%%%%%%%%%%%%%%%%%%%%%%%%%%
%\clearpage
\section{Reducible Configurations}\label{sec:reducible}

\begin{figure}[htp]
\centering
\begin{tabular}[h]{V@{\quad}V@{\quad}V@{\quad}V}
\input{figures/c1.tex}
&
\input{figures/c2.tex}&
\input{figures/c3.tex} &
\input{figures/c4.tex}  \\
\configuration{cycle4} & \configuration{cycle6} & \configuration{diamond1} & \configuration{diamond2}  \\[1ex]
\input{figures/c5.tex} &
\input{figures/c8.tex} &
\input{figures/c10.tex}&
\input{figures/c9.tex}  
 \\
\configuration{4fan} 
& 
\configuration{K6hconfig1}
&
\configuration{K6hconfig2} 
&
\configuration{K6hconfig3} \\
\input{figures/d11.tex}
&
\input{figures/d2.tex}  
&
\input{figures/3paths.tex}  
&
\input{figures/3pathsB.tex}  
\\
\configuration{d11} 
& 
\configuration{d2}
&
\configuration{3paths}
&
\configuration{3pathsB}
  \\[1ex]
\end{tabular}

\begin{tabular}[h]{W@{\quad}W@{\quad}W}
\input{figures/d1.tex} 
&
\input{figures/d9.tex} 
&
\input{figures/d7.tex}
 \\
\configuration{d1}
&
\configuration{d9}
&
\configuration{d7}
%&
%\configuration{3paths} 
\\
\input{figures/bigneedy.tex}&
\input{figures/d5.tex} &
\input{figures/d6.tex}  \\
\configuration{bigneedy} &
\configuration{d5} &
\configuration{d6} 
\\
\input{figures/d8.tex} 
&
\input{figures/d4.tex}
&
\input{figures/d4b.tex}
  \\ 
  \configuration{d8}&
 \configuration{d4} &
 \configuration{d4b}
 
 \label{lastconf} 
 \end{tabular}

\vspace{0.5em}
\textit{\footnotesize In these configurations, edges with only one endpoint are external edges. 
Vertices in $X$ are filled with white.}

\vspace{-0.5em}
%{\blue TODO: Present a key for which vertices are in $X$, what is $\ex$.}

\caption{\label{fig:configurations}Reducible configurations.}
\end{figure}

%
%
%\end{document}

In this section, we describe structures that cannot appear in a minimal counterexample to Theorem~\ref{thm:42choosableconditions}.
Let $G$ be a graph, $f : V(G) \to \mathbb{N}$, and $s$ be a nonnegative integer.
A graph is \emph{$f$-choosable} if $G$ is $L$-choosable for every list assignment $L$ where $|L(v)| \geq f(v)$.
An \emph{$(f,s)$-list-assignment} is a list assignment $L$ on $G$ such that $|L(v)| \geq f(v)$ for all $v \in V(G)$, $|L(v) \cap L(u)| \leq s$ for all edges $uv \in E(G)$, and $L(u) \cap L(v) = \emptyset$ if $uv \in E(G)$ and $f(u) = f(v) = 1$.
A graph $G$ is \emph{$(f,s)$-choosable} if $G$ is $L$-colorable for every $(f,s)$-list-assignment $L$.

\begin{definition}
A \emph{configuration} is a triple $(C, X, \ex)$ where $C$ is a plane graph, $X \subseteq V(C)$, and $\ex : V(C) \to \{ 0, 1, 2, \infty \}$ is an \emph{external degree} function.
A graph $G$ \emph{contains} the configuration $(C,X, \ex)$ if $C$ appears as an induced subgraph $C'$ of $G$, and for each vertex $v \in V(C)$, there are at most $\ex(v)$ edges in $G$ from the copy of $v$ to vertices not in $C'$.
For a triple $(C,X,\ex)$, define the \emph{list-size function} $f : V(C) \to {\mathbb N}$ as \[f(v) = \begin{cases} 4 - \ex(v) & v \in X\\ 1 & v \notin X\end{cases}.\]
A configuration $(C,X,\ex)$ is \emph{reducible} if $C$ is $(f,2)$-choosable.
\end{definition}

Note that if a graph $G$ with $(4,2)$-list assignment $L$ contains a copy of a reducible configuration $(C,X,\ex)$ and $G - X$ is $L$-choosable, then $G$ is $L$-choosable.

First, we note that if $(C,X, \ex)$ is a reducible configuration, then any way to add an edge between distinct vertices of $X$ and lower their external degree by one results in another reducible configuration.

\begin{lemma}\label{lma:iteration}
Let $(C, X, \ex)$ be a reducible configuration, and suppose that $x, y \in X$ are nonadjacent vertices with $\ex(x), \ex(y) \geq 1$.
Let $(C', X', \ex')$ be the configuration where $C' = C + xy$, $X' = X$, and $\ex'(v) = \begin{cases} \ex(v) & v \notin\{x,y\}\\ \ex(v) - 1 & v \in \{x,y\},\end{cases}$.
Then the configuration $(C', X', \ex')$ is reducible.
\end{lemma}

\begin{proof}
Let $f$ be the list-size function for $C$ and note that $C$ is $(f,2)$-choosable.
Similarly let $f'$ be the list-size function on the configuration $(C',X',\ex')$, and let $L'$ be an $(f',2)$-list assignment on $V(C')$.
Note that $f'(x) = f(x) + 1$ and $f'(y) = f(y)+1$.
Let $S = L'(x) \cap L'(y)$.
If $|S| < 2$, then add at most one element from each of $L'(x)$ and $L'(y)$ to $S$ until $|S| = 2$.
Now let $S = \{a, b\}$ such that $a \in L'(x)$ and $b\in L'(y)$, and define a list assignment $L$ on $C$ by removing $a$ from $L'(x)$ and removing $b$ from $L'(y)$.
Observe that $L$ is an $(f,2)$-list assignment and hence there exists an $L$-coloring of $C$.
Since $L(x) \cap L(y) = \emptyset$, this proper $L$-coloring of $C$ is also an $L'$-coloring of $C'$.
\end{proof}

We will use Lemma~\ref{lma:iteration} implicitly by assuming that $C[X]$ appears as an induced subgraph in our minimal counterexample $G$.

\subsection{Reducibility Proofs}

In this section, we prove that configurations \ref{cycle4}--\ref{lastconf} shown in Figure \ref{fig:configurations}  are reducible. 
%These proofs are independent of the discharging arguments in Sections~\ref{sec:5cycles}--\ref{sec:7cycles} and thus these proofs could be read in any order.

\subsubsection{Alon-Tarsi Theorem}

We will use the celebrated Alon-Tarsi Theorem~\cite{alon} to quickly prove that many of our configurations are reducible.
In fact, configurations that are demonstrated in this way are reducible for 4-choosability, not just $(4,2)$-choosability.

A digraph $D$ is an \emph{orientation} of a graph $G$ if $G$ is the underlying undirected graph of $D$ and $D$ has no 2-cycles; let $d_D^+(v)$ and $d_D^-(v)$ be the out- and in-degree of a vertex $v$ in $D$.
An \emph{Eulerian subgraph} of a digraph $D$ is a subset $S \subseteq E(D)$ such that, for every vertex $v \in V(D)$, the number of outgoing edges of $v$ in $S$ is equal to the number of incoming edges of $v$ in $S$.
Let $EE(D)$ be the number of Eulerian subgraphs of even size and $EO(D)$ be the number of Eulerian subgraphs of odd size.

\begin{figure}[tp]
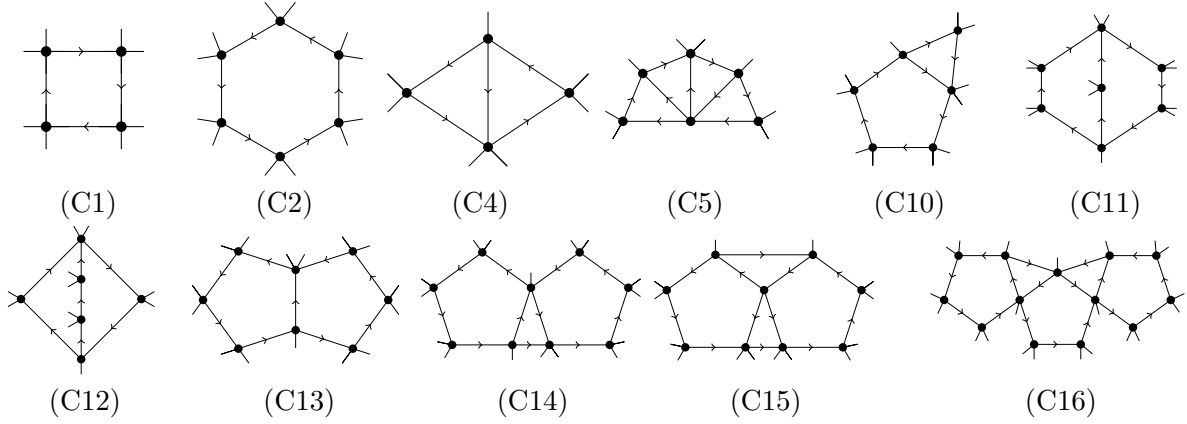

\centering

\begin{tabular}[h]{I@{\ }I@{\ }I@{\ }A@{\ }I@{\ }I}
\scalebox{1}{\input{figures/c1alontarsi.tex}}
&
\scalebox{0.9}{\input{figures/c2alontarsi.tex}}
&
\scalebox{0.8}{\input{figures/c4alontarsi.tex}}
&
\scalebox{0.9}{\input{figures/c5alontarsi.tex}}
%&
%\scalebox{0.8}{\input{figures/c7alontarsi.tex}}
&
\scalebox{0.8}{\input{figures/d2alontarsi.tex}}
&
\scalebox{0.8}{\input{figures/3pathsalontarsi.tex}}
\\
\ref{cycle4}
&
\ref{cycle6}
&
\ref{diamond2}
&
\ref{4fan}
%&
%\ref{4wheelplus}
&
\ref{d2}
&
\ref{3paths}
\end{tabular}
\\

\begin{tabular}[h]{I@{\ }A@{\ }A@{\ }A@{\ }W@{\ }I@{\ }I}
\scalebox{0.8}{\input{figures/3pathsBalontarsi.tex}}
&
\scalebox{0.8}{\input{figures/d1alontarsi.tex}}
&
\scalebox{0.8}{\input{figures/d9alontarsi.tex}}
&
\scalebox{0.8}{\input{figures/d7alontarsi.tex}}
&
\scalebox{0.8}{\input{figures/bigneedy-alontarsi.tex}}
\\
\ref{3pathsB}
&
\ref{d1}
&
\ref{d9}
&
\ref{d7}
&
\ref{bigneedy}
\end{tabular}

\caption{\label{fig:alontarsireducibility}Alon-Tarsi Orientations.}
\end{figure}

\begin{theorem}[Alon-Tarsi Theorem~\cite{alon}]\label{thm:alontarsi}
Let $G$ be a graph and $f : V(G) \to \mathbb{N}$ a function.
Suppose that there exists an orientation $D$ of $G$ such that $d_D^+(v) \leq f(v) -1$ for every vertex $v \in V(G)$ and $EE(D) \neq EO(D)$.
Then $G$ is $f$-choosable.
\end{theorem}
We call an orientation an \emph{Alon-Tarsi orientation} if it satisfies the hypotheses of Theorem~\ref{thm:alontarsi}.
For a configuration $(C,X,\ex)$ and the associated list-size function $f$, it suffices to demonstrate an Alon-Tarsi orientation of $C$ with respect to $f$.
See Figure~\ref{fig:alontarsireducibility} for a list of Alon-Tarsi orientations of several configurations.

\begin{corollary}\label{cor:specific}
The following configurations have Alon-Tarsi orientations and hence are reducible:
\begin{center}
\ref{cycle4}, \ref{cycle6}, \ref{diamond2}, \ref{4fan}, %\ref{4wheelplus}, 
 \ref{d2}, 
 \ref{3paths}, \ref{3pathsB}, 
 \ref{d1}, \ref{d9}, \ref{d7}, \ref{bigneedy}.
\end{center}
\end{corollary}

\subsubsection{Direct Proofs}

In the proofs below, we consider a configuration $(C,X,\ex)$ with list-size function $f$ and assume that an $(f,2)$-list-assignment $L$ is given for $C$.
We will demonstrate that each $C$ is $L$-colorable.
Refer to Figure~\ref{fig:configurations} for drawings of the configurations.

First recall the following fact about list-coloring odd cycles.

\begin{fact}\label{cyctri}
If $L$ is a 2-list assignment of an odd cycle, then there does not exist an $L$-coloring of the cycle if and only if all of the lists are identical. 
\end{fact}

\begin{lemma}
\ref{diamond1} is a reducible configuration.
\end{lemma}

\begin{proof}
Let  $v_1,\dots,v_4$ be the vertices of a 4-cycle with chord $v_2v_4$ and let $v_2$ and $v_4$ have external degree 1; the colors $c(v_1)$ and $c(v_3)$ are fixed.
Each of $v_2$ and $v_4$ have at least one color in their lists other than $c(v_1)$ and $c(v_3)$.
Since $|L(v_i)| \geq 3$ for each $i \in \{2,4\}$, either one of these vertices has at least two colors available, or $L(v_2) \cap L(v_4) = \{ c(v_1), c(v_3)\}$.
In either case, we can extend the coloring. 
\end{proof}

For the configurations \ref{K6hconfig1}, \ref{K6hconfig2}, and \ref{K6hconfig3}, label the vertices as in Figure~\ref{fig:k6hlabels}: label the center vertex $v_0$ and the outer vertices $v_1, \dots, v_5$, starting with the vertex directly above $v_0$, moving clockwise.

\begin{figure}[htp]
\centering
\begin{tabular}[h]{V@{\quad}V@{\quad}V}
\begin{tikzpicture}[vtx/.style={shape=coordinate}]
	\node[vtx,label=below:$v_0$] (a) at (0,0) {};
	\node[vtx,label=above:$v_2$] (b) at (18:1) {};
	\node[vtx,label=right:$v_1$] (c) at (90:1) {};
	\node[vtx,label=above:$v_5$] (d) at (162:1) {};
	\node[vtx,label=below left:$v_4$] (e) at (234:1) {};
	\node[vtx,label=below right:$v_3$] (f) at (306:1) {};
	\foreach \p in {a,b,c,d,e,f}
		\fill (\p) circle (2pt);

	\draw (a) -- (b) -- (c) -- (d) -- (e) -- (f) -- (b);
	\draw (c) -- (a) -- (d);
	\draw (e) -- (a) -- (f);

	\draw (b) -- (18:1.25);
	\draw (c) -- (90:1.25);
	\draw (100:1.25) -- (c) -- (80:1.25);
	\draw (d) -- (162:1.25);
	\draw (f) -- (306:1.25);
	\draw (e) -- (234:1.25);
	\draw (224:1.25) -- (e) -- (244:1.25);
\end{tikzpicture} 
&
\begin{tikzpicture}[vtx/.style={shape=coordinate}]
	% tikz here
	\node[vtx,label=below:$v_0$] (a) at (0,0) {};
	\node[vtx,label=above:$v_2$] (b) at (18:1) {};
	\node[vtx,label=right:$v_1$] (c) at (90:1) {};
	\node[vtx,label=above:$v_5$] (d) at (162:1) {};
	\node[vtx,label=below left:$v_4$] (e) at (234:1) {};
	\node[vtx,label=below right:$v_3$] (f) at (306:1) {};

	\foreach \p in {a,b,c,d,e,f}
		\fill (\p) circle (2pt);

	\draw (a) -- (b) -- (c) -- (d) -- (e) -- (f) -- (b);
	\draw (c) -- (a) -- (d);
	\draw (e) -- (a) -- (f);

	\draw (18:1.30) -- (b) -- (8:1.25);
	\draw (80:1.25) -- (c) -- (100:1.25);
	\draw (d) -- (162:1.25);
	\draw (f) -- (306:1.25);
	\draw (224:1.25) -- (e) -- (244:1.25);
\end{tikzpicture}
&
\begin{tikzpicture}[vtx/.style={shape=coordinate}]
	% tikz here
	\node[vtx,label=below:$v_0$] (a) at (0,0) {};
	\node[vtx,label=above:$v_2$] (b) at (18:1) {};
	\node[vtx,label=right:$v_1$] (c) at (90:1) {};
	\node[vtx,label=above:$v_5$] (d) at (162:1) {};
	\node[vtx,label=below left:$v_4$] (e) at (234:1) {};
	\node[vtx,label=below right:$v_3$] (f) at (306:1) {};

	\foreach \p in {a,b,c,d,e,f}
		\fill (\p) circle (2pt);

	\draw (a) -- (b) -- (c) -- (d) -- (e) -- (f) -- (b);
	\draw (c) -- (a) -- (d);
	\draw (e) -- (a) -- (f);

	\draw (18:1.30) -- (b) -- (8:1.25);
	\draw (80:1.25) -- (c) -- (100:1.25);
	\draw (d) -- (162:1.25);
	\draw (296:1.25) -- (f) -- (316:1.25);
	\draw (e) -- (234:1.25);
\end{tikzpicture}
\\
\ref{K6hconfig1} & \ref{K6hconfig2} &\ref{K6hconfig3}
\end{tabular}
\caption{\label{fig:k6hlabels}Vertex labels for configurations \ref{K6hconfig1}, \ref{K6hconfig2}, and \ref{K6hconfig3}.}
\end{figure}
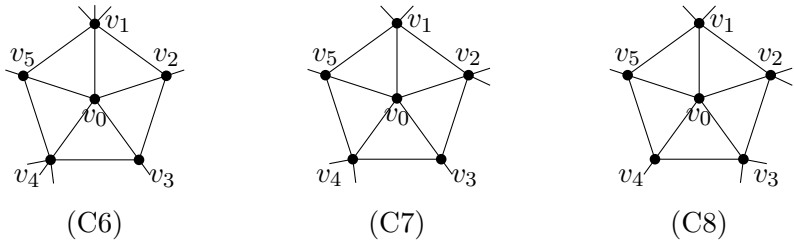

\begin{lemma}
\ref{K6hconfig1} is a reducible configuration.
\end{lemma}

\begin{proof}
The colors $c(v_1)$ and $c(v_4)$ are determined.
If $c(v_1)$ and $c(v_4)$ are both in $L(v_0)$, then select $c(v_5)$ from $L(v_5) \setminus \left(L(v_0) \cup \{c(v_1),c(v_4)\}\right)$; otherwise, select $c(v_5) \in L(v_5) \setminus \{ c(v_1), c(v_4)\}$ arbitrarily. 
Define $L'(v_0) = L(v_0) \setminus \{ c(v_1), c(v_4), c(v_5)\}$, $L'(v_2) = L(v_2) \setminus \{c(v_1)\}$, and $L'(v_3) = L(v_3)\setminus \{c(v_4)\}$ and note that $|L'(v_i)| \geq 2$ for all $i \in \{0,2,3\}$.
If $|L'(v_0)| = |L'(v_2)| = 2$, then $L'(v_0) \neq L'(v_2)$, so the 3-cycle $v_0v_2v_3$ has an $L'$-coloring by Fact~\ref{cyctri}. 
\end{proof}

\begin{lemma}
\ref{K6hconfig2} is a reducible configuration.
\end{lemma}

\begin{proof}
If there exists a color $a \in L(v_1) \cap L(v_4)$, start by assigning $c(v_1) = c(v_4) = a$; then greedily color the remaining vertices in the following order: $v_2, v_3, v_0, v_5$. 
Otherwise, $L(v_4) \cap L(v_1) = \emptyset$.

Suppose that $L(v_1) \cap L(v_5) = \emptyset$.
Select a color $c(v_4) \in L(v_4)$. 
Considering $v_4$ as an external vertex and ignoring the edge $v_1v_5$, the 4-cycle $v_0v_1v_2v_3$ forms a copy of \ref{diamond2}, which is reducible by Corollary~\ref{cor:specific}.
Thus, there exists an $L$-coloring of $v_0,\dots,v_4$; this coloring extends to $v_5$ since $L(v_1) \cap L(v_5) = \emptyset$.
If $L(v_4) \cap L(v_5) = \emptyset$, then there exists an $L$-coloring by a symmetric argument.

Otherwise, there exist colors $a \in L(v_1) \setminus L(v_5)$ and $b \in L(v_4) \setminus L(v_5)$; assign $c(v_1) = a$ and $c(v_4) = b$.
Select $c(v_2) \in L(v_2) \setminus \{ c(v_1)\}$. 
Define $L'(v_0) = L(v_0) \setminus \{ c(v_1), c(v_2), c(v_4)\}$ and $L'(v_3) = L(v_3) \setminus \{ c(v_2), c(v_4)\}$.
Note that if $|L'(v_0)| =  |L'(v_3)| = 1$, then $L(v_0) \cap L(v_3) = \{ c(v_2), c(v_4)\}$ and hence $L'(v_0) \cap L'(v_3) = \emptyset$.
Thus, the coloring extends by greedily coloring $v_3$, $v_0$, and $v_5$.
\end{proof}

\begin{lemma}
\ref{K6hconfig3} is a reducible configuration.
\end{lemma}

\begin{proof}
If $L(v_1) \cap L(v_2) = \emptyset$, then greedily color $v_2$ and $v_3$; what remains is \ref{diamond2} and the coloring extends. 
A similar argument works if $L(v_3) \cap L(v_2) = \emptyset$.

If $L(v_1) \cap L(v_3) = \emptyset$, then $|L(v_1) \cap L(v_2)| = |L(v_3) \cap L(v_2)| = 1$.
Select $c(v_1) \in L(v_1) \setminus L(v_2)$, $c(v_3) \in L(v_3)\setminus L(v_2)$.
Define $L'(v_0) = L(v_0) \setminus \{c(v_1), c(v_3)\}$, $L'(v_4) = L(v_4) \setminus \{ c(v_3) \}$, and $L'(v_5) = L(v_5) \setminus \{c(v_2)\}$.
Observe that we can $L'$-color the 3-cycle $v_0v_4v_5$ by Fact~\ref{cyctri} and then select $c(v_2) \in L(v_2) \setminus \{ c(v_0)\}$.

If there exists a color $a \in L(v_1) \cap L(v_3)$, start by assigning $c(v_1)=c(v_3) = a$ and then assign $c(v_2) \in L(v_2) \setminus \{a\}$.
Define $L'(v_0) = L(v_0) \setminus \{a, c(v_2)\}$, $L'(v_4) = L(v_4) \setminus \{ a\}$, and $L'(v_5) = L(v_5) \setminus \{a\}$.
Observe that the 3-cycle $v_0v_4v_5$ has an $L'$-coloring by Fact~\ref{cyctri}. 
\end{proof}

\begin{lemma}
\ref{d11} is a reducible configuration.
\end{lemma}

\begin{proof}
Consider the vertex $v$ of arbitrary external degree and let $c(v)$ be the color assigned to $v$.
Let $u_1$ and $u_2$ be the two neighbors of $v$ in the configuration.
If we remove $c(v)$ from the lists on $u_1$ and $u_2$, observe that at least two colors remain in every list for every vertex of the 5-cycle.
If there is no $L$-coloring of the configuration, then Fact~\ref{cyctri} asserts that all lists have size two and contain the same colors; however, this implies that $L(u_1)=L(u_2)$ and $|L(u_1) \cap L(u_2)| = 3$, a contradiction.
\end{proof}

\subsubsection{Template Configurations}

The configurations \ref{d5}--\ref{lastconf} are special cases of general constructions called \emph{template constructions}.

Let $(C,X,\ex)$ be a configuration with vertices $u, v \in X$.
A $uv$-path $P$ is called a \emph{special $uv$-path} if all internal vertices of $P$ have degree two in $C$ and external degree two.
A $uv$-path $P$ is called an \emph{extra-special $uv$-path} if all internal vertices of $P$ have external degree two and degree two in $C$, except for a consecutive pair $xy$ where $\ex(x)=\ex(y)=1$, $d(x) = d(y) = 3$, and there is a vertex $z\notin X$ such that $z$ is a common neighbor to $x$ and $y$, and $z$ is not adjacent to any other vertices in $C$.
Using these special and extra-special paths, we can describe several configurations by the following \emph{templates} (see Figure~\ref{fig:templates}), consisting of

\begin{itemize}
\item (B1) a triangle $uvw,$ where $\ex(u) = \ex(w) = 2$, $\ex(v) = 0$,  an extra-special $uv$-path $P_1$, and a special $vw$-path $P_2$, and
%Let $u$, $v$, and $w$  be vertices in a 3-cycle with $\ex(u) = \ex(w) = 2$, $\ex(v) = 0$, an extra-special $uv$-path $P_1$, and a special $vw$-path $P_2$.

\item (B2) a triangle $vwr$, where $\ex(r) = \infty$, $\ex(w) = 1$, $\ex(v) = 0$, a vertex $u$ adjacent to $v$ where $\ex(u) = 2$, an extra-special $uv$-path $P_1$, and a special $vw$-path $P_2$.

% Let $r$, $u$, $v$, and $w$  be vertices where $uvwr$ is a 4-cycle and $vr$ is an edge with $\ex(r) = \infty$, $\ex(u) = \ex(w) = 1$, $\ex(v) = 0$,  an extra-special $uv$-path $P_1$, and a special $vw$-path $P_2$.
%\item (B3) Let $r$, $s$, $t$, $u$, $v$, $w$, $x$, $y$, $z$  be vertices where  $uvwr$ is a 4-cycle and $vr$ is an edge, $vxy$ is a 3-cycle, and $stzy$ is a 4-cycle and $ty$ is an edge, with $\ex(r) = \ex(s) = \infty$, $\ex(z) = 2$, $\ex(t) = \ex(u) = \ex(w) = \ex(x) = 1 $, $\ex(v) = \ex(y) = 0$, a special $ux$-path $P_1$, and a special $wz$-path $P_2$.
\end{itemize}

\begin{figure}[htp]
\centering
\begin{tabular}[h]{W@{\quad}W@{\quad}W} %@{\quad}W}
%\input{figures/b1.tex}
%&
%\input{figures/b2.tex}&
%\input{figures/b3.tex} 
% \\
%(B1) & (B2) & (B3)  \\
\begin{tikzpicture}[scale=1,vtx/.style={shape=coordinate}]
	% tikz here
	\node[vtx,label=above left:$u$] (La) at (0,0) {};
	\node[vtx,label=below left:$P_1$] (Lb) at (180:1) {};
	\node[vtx,label=below right:$x$] (Lc) at ($(Lb) !1! 108:(La)$) {};
	\node[vtx,label=below:$y$] (Ld) at ($(Lc) !1! 108:(Lb)$) {};

	\node[vtx,label=below left:$z$] (g) at ($(Lc) !1! 60:(Ld)$) {};

	\node[vtx,label=right:$v$] (MB) at ($(Ld) !1! 108:(Lc)$) {};
	\node[vtx] (MT) at ($(MB) +(90:1)$) {};

	\node[vtx] (Ra) at ($(MB) !1! -54:(MT)$) {};
	\node[vtx] (Rb) at ($(Ra) !1! 108:(MB)$) {};
	\node[vtx,label=below right:$P_2$] (Rc) at ($(Rb) !1! 108:(Ra)$) {};
	\node[vtx,label=above right:$w$] (Rd) at ($(Rc) !1! 108:(Rb)$) {};

	\fill (MB) circle (2pt);
	\fill (g) circle (2pt);

%		\fill (Lb) circle (2pt);
		\fill (Lc) circle (2pt);
		\fill (La) circle (2pt);
		\fill (Ld) circle (2pt);
		\fill (Rd) circle (2pt);

	\draw[style=dashed] (La) to[out=-144,in=324] (Lb) to[out=144,in=252] (Lc) -- (Ld) to[out=0,in=108] (MB);
	\draw (MB) -- (La);
	\draw (Rd) -- (MB);
	\draw[style=dashed] (MB) to[out=72,in=180] (Ra) to[out=0,in=108] (Rb) to[out=-72,in=36] (Rc) to[out=-144,in=324] (Rd);
	\draw (La) -- (Rd);
	\draw (Ld) -- (g) -- (Lc);

	\draw (La) -- +(-54:0.25) -- (La) -- +(234:0.25);
%	\draw (Lb) -- +(168:0.25) -- (Lb) -- +(240:0.25);
	\draw (Lc) -- +(180:0.25);
	\draw (Ld) -- +(78:0.25);
	
	\draw (g) -- +(171:0.3) -- (g) -- +(81:0.3);
	\draw (g) -- +(126:0.3);
%	\draw[style=dashed] (g) -- +(111:0.3);
%	\draw[style=dashed] (g) -- +(141:0.3);

%	\draw (Ra) -- +(54:0.25) -- (Ra) -- +(126:0.25);
%	\draw (Rb) -- +(36:0.25) -- (Rb) -- +(-36:0.25);
%	\draw (Rc) -- +(12:0.25) -- (Rc) -- +(-60:0.25);
	\draw (Rd) -- +(-54:0.25) -- (Rd) -- +(234:0.25);
	
	\foreach \p in {MB,La,Lc,Ld,Rd}
		\fill[fill=white] (\p) circle (1.5pt);

\end{tikzpicture}
&
\begin{tikzpicture}[scale=1,vtx/.style={shape=coordinate}]
	% tikz here
	\node[vtx,label=below:$y$] (La) at (0,0) {};
	\node[vtx,label=below:$x$] (Lb) at (180:1) {};
	\node[vtx,label=left:$P_1$] (Lc) at ($(Lb) !1! 108:(La)$) {};
	\node[vtx,label=below:$u$] (Ld) at ($(Lc) !1! 108:(Lb)$) {};

	\node[vtx,label=right:$z$] (g) at ($(La) !1! 60:(Lb)$) {};

	\node[vtx,label=right:$v$] (MB) at ($(Ld) !1! 108:(Lc)$) {};
	\node[vtx,label=below left:$r$] (MT) at ($(MB) +(90:1)$) {};

	\node[vtx,label=below:$w$] (Ra) at ($(MB) !1! -54:(MT)$) {};
	\node[vtx,label=right:$P_2$] (Rb) at ($(Ra) !1! 108:(MB)$) {};
	\node[vtx] (Rc) at ($(Rb) !1! 108:(Ra)$) {};
	\node[vtx] (Rd) at ($(Rc) !1! 108:(Rb)$) {};

	\fill (MB) circle (2pt);
	\fill (MT) circle (2pt);
	\fill (g) circle (2pt);

	\fill (La) circle (2pt);
	\fill (Lb) circle (2pt);
	\fill (Ld) circle (2pt);
	\fill (Ra) circle (2pt);

	\draw[style=dashed] (MB) to[out=-100,in=36] (La) -- (Lb) to[out=144,in=252]  (Lc) to[out=72,in=180]   (Ld);
	\draw (Ld) -- (MB);
	\draw[style=dashed] (Ra) to[out=0,in=108] (Rb) to[out=-72,in=36] (Rc) to[out=-144,in=324] (Rd) to[out=144,in=280] (MB);
	\draw (MT) -- (MB) -- (Ra) -- (MT);
	\draw (Lb) -- (g) -- (La);

	\draw (La) -- +(-18:0.25);
	\draw (Lb) -- +(198:0.25);
%	\draw (Lc) -- +(216:0.25) -- (Lc) -- +(144:0.25);
	\draw (Ld) -- +(126:0.25);
	\draw (Ld) -- +(84:0.25);
	
	\draw (g) -- +(-45:0.25) -- (g) -- +(-135:0.25);
%	\draw[style=dashed] (g) -- +(-75:0.3);
	\draw(g) -- +(-90:0.3);
	\draw (MT) -- +(180:0.3) -- (MT) -- +(135:0.3);
	\draw (MT) -- +(90:0.3);
%	\draw[style=dashed] (MT) -- +(105:0.3);

	\draw (Ra) -- +(54:0.25);
%	\draw (Rb) -- +(36:0.25) -- (Rb) -- +(-36:0.25);
%	\draw (Rc) -- +(-90:0.25) -- (Rc) -- +(-18:0.25);
%	\draw (Rd) -- +(-90:0.25) -- (Rd) -- +(-162:0.25);
	
	\foreach \p in {MB,La,Lb,Ld,Ra}
		\fill[fill=white] (\p) circle (1.5pt);

\end{tikzpicture} 
%&
%\input{figures/b6.tex} 
\\
(B1)
&
(B2)
%&
%(B3)
\end{tabular}

\vspace{0.5em}
\textit{\footnotesize Dotted lines indicate special paths or extra-special paths. 
Vertices in $X$ are filled with white.}
%{\blue TODO: Present a key for which vertices are in $X$, what is $\ex$.}

\vspace{-0.5em}

\caption{\label{fig:templates}Templates for reducible configurations.}
\end{figure}

%
%
%\end{document}

We make some basic observations about special and extra-special paths that will be used to prove that these templates correspond to reducible configurations.

Let $P$ be a special $uv$-path or an extra-special $uv$-path.
For every color $a \in L(u)$, let $g_P^u(a)$ be the set containing each color $b \in L(v)$ such that assigning $c(u) = a$ and $c(v) = b$ does not extend to an $L$-coloring of $P$.
Since we can greedily color $P$ starting at $u$ until reaching $v$, there is at most one color in $g_P^u(a)$.
Further, $g_P^u(a) \neq \emptyset$ if and only if this greedy coloring process has exactly one choice for each vertex in $P$.
Thus, if $g_P^u(a) = \{b\}$ then also $g_P^v(b) = \{a\}$.

Since $L$ is an $(f,2)$-list assignment, adjacent vertices have at most two colors in common.
Thus, there are at most two colors $a_1, a_2 \in L(u)$ such that $g_P^u(a_i) \neq \emptyset$.
Moreover, observe that if there are two distinct colors $a_1, a_2 \in L(u)$ such that $g_P^u(a_i) \neq \emptyset$, then both $a_1$ and $a_2$ are in every list along $P$ and hence $\{a_1,a_2\} \subseteq L(v)$. 

If $P$ is an extra-special $uv$-path with 3-cycle $xyz$ where $xy$ is in the path $P$, then after a color is assigned to $z$ (as $\ex(z) = \infty$) either one of $x$ or $y$ has three colors available or $|L(x)\cap L(y)| \leq 1$.
Therefore, if $P$ is an extra-special $uv$-path, then there is at most one color $a \in L(u)$ such that $g_P^u(a) \neq \emptyset$.

\begin{lemma}
All configurations matching the template (B1) are reducible.
\end{lemma}
\begin{proof}
Let $(C, X, \ex)$ be a configuration matching the template (B1) and let $L$ be an $(f,2)$-list assignment.

%Suppose that there exist two colors $a_1, a_2 \in L(v) \setminus L(u)$ where $a_1 \neq L(w)$.
%If $g_{P_1}^v(a_1) = \emptyset$,  then assign $c(v) = a_1$, select $c(w) \in L(w) \setminus g_{P_2}^v(a_1)$, then select $c(u) \in L(u) \setminus \{ c(w)\}$; the coloring extends to $P_1$ and $P_2$.
%If $g_{P_2}^v(a_1) = \emptyset$, then assign $c(v) = a_1$, select $c(u) \in L(u) \setminus g_{P_1}^v(a_1)$, then select $c(w) \in L(w) \setminus \{ c(u)\}$; the coloring extends to $P_1$ and $P_2$.
%If $g_{P_i}^v(a_1) \neq \emptyset$ for each $i \in \{1,2\}$, then we have $g_{P_i}^v(a_2) = \emptyset$ for each $i \in \{1,2\}$; hence assign $c(v) = a_2$, select $c(w) \in L(w) \setminus \{a_2\}$ and $c(u) \in L(u) \setminus \{c(w)\}$; the coloring extends to $P_1$ and $P_2$.
%
%Thus, of the two colors in $L(v) \setminus L(u)$, both are in $L(w)$.
%This implies that $L(v) = L(u) \cup L(w)$.
%Let $L(u) = \{a_1,a_2\}$ and $L(w) = \{b_1,b_2\}$.
%Since $P_1$ is an extra-special path, there is at least one $i \in \{1,2\}$  such that $g_{P_1}^v(a_i) = \emptyset$.
%Thus, assign $c(v) = a_i$, select $c(u) \in L(u) \setminus \{a_i\}$ and $c(w) \in L(w) \setminus g_{P_2}^v(a_i)$; the coloring extends to $P_1$ and $P_2$.
Let $L(u) = \{a_1,a_2\}$. 
Since $P_1$ is an extra-special path, there is at least one $i \in \{1,2\}$ such that $g_{P_1}^u(a_i) = \emptyset$.
Assign $c(u) = a_i$, select $c(w) \in L(w) \setminus \{a_i\}$ and $c(v) \in L(v) \setminus \left(\{c(u), c(w)\} \cup g_{P_1}^w(c(w))\right)$; the coloring extends to $P_1$ and $P_2$.
\end{proof}

\begin{corollary}
The configurations \ref{d5}, \ref{d6}, and \ref{d8} match the template (B1), and hence they are reducible.
\end{corollary}

\begin{lemma}
All configurations matching the template (B2) are reducible.
\end{lemma}
\begin{proof}
Let $(C, X, \ex)$ be a configuration matching the template (B2) and let $L$ be an $(f,2)$-list assignment.
Let $c(r)$ be the unique color in the list $L(r)$.
Let $L(u) = \{a_1,a_2\}$. 
Since $P_1$ is an extra-special path, there is at least one $i \in \{1,2\}$ such that $g_{P_1}^u(a_i) = \emptyset$.
Assign $c(u) = a_i$.

If $c(r) \notin L(v)$, then select $c(w) \in L(w)$, and $L(v) \in L(v) \setminus \left(\{ c(u), c(w)\} \cup g_{P_2}^w(c(w))\right)$; the coloring extends to $P_1$ and $P_2$.

If $c(r) \in L(v)$, then select $c(w) \in L(w) \setminus L(v)$; observe $c(w) \neq c(r)$.
There exists a color $c(v) \in L(v) \setminus \left( \{ c(r), c(u) \} \cup g_{P_2}^w(c(w))\right)$; the coloring extends to $P_1$ and $P_2$.
%If $|L(u)| = |L(w)| = 2$ and $|L(v)| = 3$, then $|L(u)\cap L(v)| \leq 1$ and $|L(w)\cap L(v)| \leq 1$.
%Thus, there exists a color $a \in L(v) \setminus (L(u) \cup L(w))$.
%Assign $c(v) = a$, select $c(u) \in L(u) \setminus g_{P_1}^v(a)$ and $c(w) \in L(w)\setminus g_{P_2}^v(a)$; the coloring extends to $P_1$ and $P_2$.
%Thus, we can assume that the color $c(r)$ does not appear in one of $L(u)$, $L(v)$, or $L(w)$.
%
%If $|L(u)| = 3$, then select $a \in L(v) \setminus L(w)$ and assign $c(v) = a$.
%Select $c(u) \in L(u) \setminus (\{a\} \cup g_{P_1}^v(a))$ and $c(w) \in L(w) \setminus g_{P_2}^v(a)$; the coloring extends to $P_1$ and $P_2$.
%The case where $|L(w)| = 3$ is symmetric.
%
%Therefore, $|L(u)| = |L(w)| = 2$ and $|L(v)| = 4$.
%If there exists a color $a \in L(v) \setminus (L(u)\cup L(w))$, then we can color $C$ as in the first case above.
%Hence $L(v) = L(u) \cup L(w)$ and $L(u) \cap L(w) = \emptyset$.
%Since $P_1$ is an extra-special $uv$-path, there exists at least one color $a \in L(v) \cap L(u)$ such that $g_{P_1}^v(a) = \emptyset$.
%Thus let $c(v) = a$, select $c(u) \in L(u) \setminus \{a\}$ and $c(w) \in L(w) \setminus g_{P_2}^v(a)$; the coloring extends to $P_1$ and $P_2$.
\end{proof}

\begin{corollary}
Using Lemma~\ref{lma:iteration}, the configurations \ref{d4} and \ref{d4b} match the template (B2), and hence they are reducible.
\end{corollary}

\section{No Chorded 5-Cycle}\label{sec:5cycles}

In this section we show the case of forbidding chorded 5-cycles from Theorem~\ref{thm:42choosableconditions}. 

\begin{theorem}\label{thm:5cycles}
If $G$ is a plane graph not containing a chorded 5-cycle, then $G$ is $(4,2)$-choosable.
\end{theorem}

\begin{proof}
Let $G$ be a counterexample minimizing $n(G)$ among all plane graphs avoiding chorded 5-cycles with a $(4,2)$-list assignment $L$ such that $G$ is not $L$-choosable.
Observe that $n(G) \geq 4$; in fact, $\delta(G) \geq 4$.
Since $G$ is a minimal counterexample, $G$ does not contain any of the reducible configurations % \ref{cycle4}, \ref{cycle6}, or 
\ref{d11}--\ref{lastconf}.
If $(C,X,\ex)$ is a reducible configuration, then by Lemma~\ref{lma:iteration} $C$ does not appear as a subgraph of $G$ where $d_G(x) \leq d_C(x)+\ex(x)$ for all $x\in V(C)$.
Further, the configurations \ref{d1}--\ref{d4b} are large enough that we must consider configurations that are formed by identifying certain pairs of vertices in these configurations. 
In Appendix~\ref{appx:reducible}, we concretely check all vertex pairs that avoid creating a chorded 5-cycle and find that all resulting configurations are reducible.

For each $v \in V(G)$ and $f \in F(G)$ define initial charges $\mu(v) = d(v) - 6$ and $\nu(f) = 2 \ell(f) - 6$.
By Euler's Formula, the sum of initial charges is $-12$.
After charges are initially assigned, the only elements with negative charge are 4-vertices and 5-vertices.
Since chorded 5-cycles are forbidden, there is no 3-fan in $G$ and every 4-face is adjacent to only $4^+$-faces. 
The possible arrangements of 3-, $4^+$-, or $5^+$-faces incident to 4- and 5-vertices are shown in Figure~\ref{R1R2config}.

\begin{figure}[h]
\centering
\begin{tikzpicture}[line cap=round,line join=round,>=triangle 45,x=1.0cm,y=1.0cm]
%\clip(-0.74,-3.36) rectangle (13.7,3.52);
\draw [fill=black] (2.,2.) circle (1.5pt);
\draw[color=black] (2.14,2.15) node {$v$};
\draw (2.,2.)-- (1.,2.); \draw [fill=white] (1.,2.) circle (1.5pt);
\draw (2.,2.)-- (2.,3.); \draw [fill=white] (2.,3.) circle (1.5pt);
\draw (2.,1.)-- (2.,2.); \draw [fill=white] (2.,1.) circle (1.5pt);
\draw (3.,2.)-- (2.,2.); \draw [fill=white] (3.,2.) circle (1.5pt);
\draw (4.,2.)-- (5.,2.); 
\draw (5.,2.)-- (5.,3.);
\draw (5.,3.)-- (4.,2.);
\draw (5.,2.)-- (6.,2.); \draw [fill=white] (6.,2.) circle (1.5pt); \draw [fill=white] (4.,2.) circle (1.5pt);
\draw (5.,2.)-- (5.,1.); \draw [fill=white] (5.,1.) circle (1.5pt); \draw [fill=white] (5.,3.) circle (1.5pt);
\draw (8.,3.)-- (7.,2.);
\draw (7.,2.)-- (8.,2.); 
\draw (8.,2.)-- (8.,3.); %\draw [fill=black] (3.,2.) circle (1.5pt);
\draw (9.,2.)-- (8.,3.); 
\draw (8.,2.)-- (9.,2.);% \draw [fill=black] (.,2.) circle (1.5pt);
\draw (8.,2.)-- (8.,1.); \draw [fill=white] (8.,1.) circle (1.5pt);\draw [fill=white] (9.,2.) circle (1.5pt);\draw [fill=white] (7.,2.) circle (1.5pt); \draw [fill=white] (8.,3.) circle (1.5pt);
\draw (10.,2.)-- (11.,2.);
\draw (11.,2.)-- (11.,3.);
\draw (11.,3.)-- (10.,2.); %\draw [fill=black] (3.,2.) circle (1.5pt);
\draw (11.,2.)-- (11.,1.); 
\draw (11.,1.)-- (12.,2.);
\draw (12.,2.)-- (11.,2.); %\draw [fill=black] (3.,2.) circle (1.5pt);
 \draw [fill=white] (10.,2.) circle (1.5pt);
  \draw [fill=white] (12.,2.) circle (1.5pt);
  \draw [fill=white] (11.,1.) circle (1.5pt);
   \draw [fill=white] (11.,3.) circle (1.5pt);
\draw (2,0.75) node[anchor=north] {(a) $4^+\,4^+\,4^+\,4^+$};
\draw (5,0.75) node[anchor=north] {(b) $3\,5^+\,4^+\,5^+$};
\draw (8,0.75) node[anchor=north] {(c) $3\,3\,5^+\,5^+$};
\draw (11,0.75) node[anchor=north] {(d) $3\,5^+\,3\,5^+$};
\draw (0.5,-1.5)-- (0.5,-0.5);
\draw (0.5,-1.5)-- (-0.549,-1.191);
\draw (0.5,-1.5)-- (-0.088,-2.309);
\draw (0.5,-1.5)-- (1.088,-2.309);
\draw (0.5,-1.5)-- (1.451,-1.191);
\draw (3.5,-1.5)-- (2.912,-2.309);
\draw (3.5,-1.5)-- (4.088,-2.309);
\draw (3.5,-1.5)-- (2.549,-1.191);
\draw (3.5,-0.5)-- (3.5,-1.5);
\draw (3.5,-1.5)-- (4.451,-1.191);
\draw (2.912,-2.309)-- (4.088,-2.309);
\draw (6.5,-0.5)-- (6.5,-1.5);
\draw (6.5,-1.5)-- (5.549,-1.191);
\draw (5.549,-1.191)-- (5.912,-2.309);
\draw (5.912,-2.309)-- (6.5,-1.5);
\draw (6.5,-1.5)-- (7.088,-2.309);
\draw (7.088,-2.309)-- (7.451,-1.191);
\draw (7.451,-1.191)-- (6.5,-1.5);
\draw (9.5,-0.5)-- (8.549,-1.191);
\draw (8.549,-1.191)-- (9.5,-1.5);
\draw (9.5,-1.5)-- (9.5,-0.5);
\draw (9.5,-0.5)-- (10.451,-1.191);
\draw (10.451,-1.191)-- (9.5,-1.5);
\draw (9.5,-1.5)-- (8.912,-2.309);
\draw (10.088,-2.309)-- (9.5,-1.5);
\draw (12.5,-0.5)-- (11.549,-1.191);
\draw (11.549,-1.191)-- (12.5,-1.5);
\draw (12.5,-1.5)-- (12.5,-0.5);
\draw (12.5,-0.5)-- (13.451,-1.191);
\draw (13.451,-1.191)-- (12.5,-1.5);
\draw (12.5,-1.5)-- (11.912,-2.309);
\draw (11.912,-2.309)-- (13.088,-2.309);
\draw (13.088,-2.309)-- (12.5,-1.5);
\draw [fill=white] (0.5,-0.5) circle (1.5pt);
\draw [fill=white] (-0.549,-1.191) circle (1.5pt);
\draw [fill=white] (-0.088,-2.309) circle (1.5pt);
\draw [fill=white] (1.088,-2.309) circle (1.5pt);
\draw [fill=white] (1.451,-1.191) circle (1.5pt);

\draw [fill=white] (2.912,-2.309) circle (1.5pt);
\draw [fill=white] (4.088,-2.309) circle (1.5pt);
\draw [fill=white] (2.549,-1.191) circle (1.5pt);
\draw [fill=white] (4.451,-1.191) circle (1.5pt);
\draw [fill=white] (3.5,-0.5) circle (1.5pt);

\draw [fill=white] (5.912,-2.309) circle (1.5pt);
\draw [fill=white] (7.088,-2.309) circle (1.5pt);
\draw [fill=white] (5.549,-1.191) circle (1.5pt);
\draw [fill=white] (7.451,-1.191) circle (1.5pt);
\draw [fill=white] (6.5,-0.5) circle (1.5pt);

\draw [fill=white] (8.912,-2.309) circle (1.5pt);
\draw [fill=white] (10.088,-2.309) circle (1.5pt);
\draw [fill=white] (8.549,-1.191) circle (1.5pt);
\draw [fill=white] (10.451,-1.191) circle (1.5pt);
\draw [fill=white] (9.5,-0.5) circle (1.5pt);

\draw [fill=white] (11.912,-2.309) circle (1.5pt);
\draw [fill=white] (13.088,-2.309) circle (1.5pt);
\draw [fill=white] (11.549,-1.191) circle (1.5pt);
\draw [fill=white] (13.451,-1.191) circle (1.5pt);
\draw [fill=white] (12.5,-0.5) circle (1.5pt);

\draw (0.5,-2.5) node[anchor=north] {(e) $4^+\,4^+\,4^+\,4^+\,4^+$};
\draw (3.5,-2.5) node[anchor=north] {(f) $3\,5^+\,4^+\,4^+\,5^+$};
\draw (6.5,-2.5) node[anchor=north] {(g) $3\,5^+\,3\,5^+\,5^+$};
\draw (9.5,-2.5) node[anchor=north] {(h) $3\,3\,5^+\,4^+\,5^+$};
\draw (12.5,-2.5) node[anchor=north] {(i) $3\,3\,5^+\,3\,5^+$};
\begin{scriptsize}
\draw [fill=black] (5.,2.) circle (1.5pt);
\draw[color=black] (5.14,2.15) node {$v$};
\draw [fill=black] (8.,2.) circle (1.5pt);
\draw[color=black] (8.14,2.15) node {$v$};
\draw [fill=black] (11.,2.) circle (1.5pt);
\draw[color=black] (11.14,2.15) node {$v$};
\draw [fill=black] (0.5,-1.5) circle (1.5pt);
\draw[color=black] (0.64,-1.3) node {$v$};
\draw [fill=black] (3.5,-1.5) circle (1.5pt);
\draw[color=black] (3.64,-1.3) node {$v$};
\draw [fill=black] (6.5,-1.5) circle (1.5pt);
\draw[color=black] (6.64,-1.3) node {$v$};
\draw [fill=black] (9.5,-1.5) circle (1.5pt);
\draw[color=black] (9.64,-1.3) node {$v$};
\draw [fill=black] (12.5,-1.5) circle (1.5pt);
\draw[color=black] (12.64,-1.3) node {$v$};
\end{scriptsize}
\end{tikzpicture}
\caption{\label{R1R2config} Possible cyclic arrangements of $3$-, $4^+$-, and $5^+$-faces incident to 4- and 5-vertices}
\end{figure}
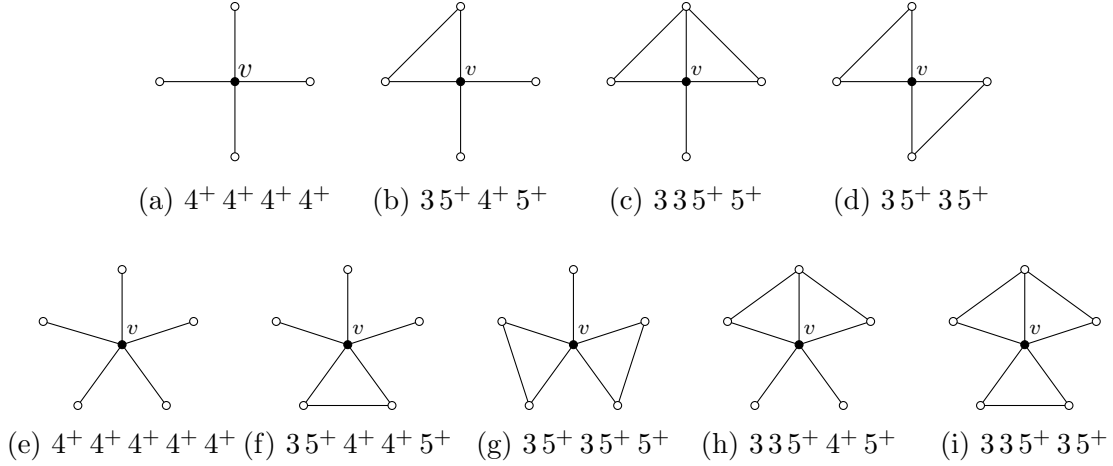

Sequentially apply the following discharging rules.  
Note that, for a vertex $v$ and a face $f$, we define $\mu_i (v)$ and $\nu_i (f)$ to be the charge on $v$ and $f$, respectively, after applying rule (R$i$).

\begin{itemize}
\item[(R1)] Let $v$ be a 4-vertex and $f$ be a $4^+$-face incident to $v$.
If $f$ is adjacent to a 3-face that is also incident to $v$, then $f$ sends charge $1$ to $v$; otherwise, $f$ sends charge $\frac{1}{2}$ to $v$.

\item[(R2)] Let $v$ be a 5-vertex.  
If $f$ is a $4^+$-face incident to $v$, then $f$ sends charge $\frac12$ to $v$.
\end{itemize}

%\begin{definition}
A face $f$ is a \emph{needy} face if $\nu_2 (f) < 0$; otherwise, $f$ is \emph{non-needy}.
%\end{definition}

\begin{itemize}
\item[(R3)] If $v$ is a 5-vertex incident to a needy 5-face $f$, then $v$ sends charge $\frac12$ to $f$.
\end{itemize}

%\begin{definition}
A vertex $v$ is a \emph{needy} vertex if $\mu_3 (v) < 0$; otherwise, $v$ is \emph{non-needy}.
%\end{definition}

\begin{itemize}
\item[(R4)] If $f$ is a non-needy $5^+$-face incident to a needy 5-vertex $v$, then $f$ sends charge $\frac12$ to $v$.
\end{itemize}

%%\begin{definition}
%A face $f$ is a \emph{second-hand-needy} face if $\nu_4 (f) < 0$.
%%\end{definition}
%
%\begin{itemize}
%\item[(R5)] If $v$ is a non-needy 5-vertex incident to a second-hand-needy 5-face $f$, then $v$ sends charge $\frac12$ to $f$.
%%{\blue Is there a more concise way to state (R3)-(R5)?}
%\end{itemize}

We show that $\mu_4 (v) \geq 0$ for each vertex $v$ and $\nu_4 (f) \geq 0$ for each face $f$.
Since the total charge was preserved during the discharging rules, this contradicts the negative charge sum from the initial charge values.
We begin by considering the charge distribution after applying (R1) and (R2).

Let $v$ be a vertex.
If $v$ is a 4-vertex, then $\mu(v) = -2$ and $v$ receives total charge at least 2 from its neighboring faces by (R1).
Furthermore, $v$ is not affected by any rules after (R1), so $\mu_4 (v) \geq 0$.
If $v$ is a $6^+$-vertex, then $\mu(v) \geq 0$ and $v$ is not affected by any other rules, so $\mu_4 (v) \geq 0$.
If $v$ is a 5-vertex, then $\mu(v) = -1$ and $v$ receives  total charge at least 1 from its neighboring faces by (R2).
Therefore, for any vertex $v$, $\mu_2 (v) \geq 0$.

Let $f$ be a face.
If $f$ is a 3-face, then $\nu(f) = 0$ and $f$ is not affected by any rule, so $\nu_4 (f) = 0$.
If $f$ is a 4-face, then $\nu(f) = 2$.
In (R1) and (R2), the only faces that send charge 1 to a single vertex are adjacent to a 3-face.
A 4-face adjacent to a 3-face is a chorded 5-cycle, which is forbidden by assumption, so $f$ sends charge at most $\frac12$ to each vertex.
Since 4-faces are not affected by rules (R3)--(R4), $\nu_4 (f) \geq 0$.
If $f$ is a $6^+$-face, then $f$ has at least as much initial charge as it has incident vertices.
If $v$ is a 4-vertex incident to $f$, then $f$ sends charge at most 1 to $v$ by (R1) and does not send any charge to $v$ by rules (R2)--(R4).
If $v$ is a 5-vertex incident to $f$, then $f$ sends charge $\frac12$  to $v$ by (R1), and possibly another charge $\frac12$ by (R4), and does not send charge to $v$ by (R1) or (R3).
Thus $f$ sends charge at most 1 to each incident vertex, and $\nu_4 (f) \geq 0$.

If $f$ is a 5-face, then $\nu(f) = 4$ and $f$ sends charge at most 1 to each incident vertex by (R1) and (R2).
Observe that if $\nu_2 (f) = -1$, then $f$ is incident to five 4-vertices and $f$ is adjacent to at least one 3-face; this forms \ref{d11}, a contradiction.
Therefore, we have the following claim about the structure of a needy 5-vertex.

\begin{claim}
If $f$ is a needy 5-face, then $\nu_2(f) = -\frac12$ and $f$ is adjacent to exactly one $5$-vertex.
\end{claim}

We now consider the charge distribution after applying (R3).
If $f$ is a needy 5-face, then $\nu_2 (f) = -\frac12$ and $f$ is adjacent to exactly one 5-vertex, so $\nu_3 (f) = 0$.
No faces lose charge in (R3), therefore $\nu_3 (f) \geq 0$ for any face $f$.

\begin{claim}
If $v$ is a needy 5-vertex, then $v$ is incident to three 3-faces, two $4^+$-faces, and exactly one needy 5-face; hence $\mu_3(v) = -\frac{1}{2}$.
\end{claim}

\begin{proof}
Suppose that $v$ is a vertex such that $\mu_3 (v) < 0$, and consider the cyclic arrangement of 3- and $4^+$-faces about $v$.

\begin{mycases}
\mycase{$v$ is incident to at least four $4^+$-faces (Figures~\ref{R1R2config}(e) and \ref{R1R2config}(f)).}
Since $\mu_2(v) \geq 1$ and $\mu_3(v) < 0$, $v$ is incident to at least three needy 5-faces.
Hence two of the needy 5-faces are adjacent, forming \ref{d1}, a contradiction.

\mycase{$v$ is incident to two non-adjacent $3$-faces and three $4^+$-faces (Figure~\ref{R1R2config}(g)).}
Since $\mu_2(v) = \frac{1}{2}$ and $\mu_3(v) < 0$, $v$ is incident to two needy 5-faces, $f_1$ and $f_2$.
If these two faces are adjacent, then they form \ref{d1}, a contradiction.  
Otherwise, they share a 3-face $t$ as a neighbor and all vertices incident to $f_1$, $f_2$, and $t$ other than $v$ are 4-vertices, so the vertices incident to $f_1$ and $t$ form \ref{d2}, a contradiction.

\mycase{$v$ is incident to two adjacent $3$-faces and three $4^+$-faces (Figure~\ref{R1R2config}(h)).}
Since $\mu_2(v) = \frac{1}{2}$ and $\mu_3(v) < 0$, $v$ is incident to two needy 5-faces, $f_1$ and $f_2$.
If $f_1$ and $f_2$ are adjacent then they form \ref{d1}, a contradiction.
Thus, $f_1$ and $f_2$ are not adjacent, but they are each adjacent to a 3-face incident to $v$.
Since $f_i$ is needy for each $i \in \{1,2\}$, $f_i$ sent charge 1 to every 4-vertex incident to $f_i$.
By (R1), every 4-vertex incident to $f_i$ is incident to a 3-face adjacent to $f_i$.  
Therefore, $f_1$ is adjacent to a 3-face that does not share any vertices with the the two 3-faces incident to $v$, forming one of \ref{d4} or \ref{d4b}, a contradiction.

\mycase{$v$ is incident to three $3$-faces and two $4^+$-faces (Figure~\ref{R1R2config}(i)).}
If $v$ is incident to two needy 5-faces $f_1$ and $f_2$, then the 3-face $t$ adjacent to both $f_1$ and $f_2$ is incident to two 4-vertices, and the vertices incident to $f_1$ and $t$ form \ref{d2}, a contradiction.  
Therefore, $v$ is incident to exactly one needy 5-face, as claimed. \qedhere
\end{mycases}
\end{proof}

By (R4), every needy 5-vertex receives charge $\frac12$ from its unique incident non-needy $5^+$-face, so $\mu_4 (v) \geq 0$ for every vertex $v$.
Each needy 5-face has nonnegative charge after (R3), so if $\nu_4(f) < 0$ for some $5$-face $f$, then $f$ sends charge by (R4), and thus is non-needy.

\begin{figure}[h]
\centering

\mbox{
\subfigure[\label{2needy}A 5-face $f$ with $\nu_4(f) < 0$.]{
\begin{tikzpicture}[scale=0.66,line cap=round,line join=round,>=triangle 45,x=2.0cm,y=2.0cm]
%\clip(1.36,1.62) rectangle (5.4,4.42); 
\draw (2.022,2.208)-- (1.917,3.202); \draw [fill=black]  (2.022,2.208) circle (2.5pt);
\draw (1.917,3.202)-- (2.831,3.609);  \draw [fill=black]  (1.917,3.202) circle (2.5pt);
\draw (2.831,3.609)-- (3.5,2.866);  \draw [fill=black] (2.831,3.609) circle (2.5pt);
\draw (3.5,2.866)-- (3.,2.); \draw [fill=black] (3,2) circle (2.5pt);
\draw (3.,2.)-- (2.022,2.208);
\draw (3.,2.)-- (4,2.);
\draw (3.5,2.866)-- (4,2.);
\draw (4,2.)-- (4.978,2.208);
\draw (4.978,2.208)-- (5.083,3.202);
\draw (4.169,3.609)-- (5.083,3.202);
\draw (4.169,3.609)-- (3.5,2.866);
\draw (3.5,2.866)-- (3.5,4.352);
\draw (3.5,4.352)-- (4.169,3.609);\draw [fill=black] (3.5,4.352) circle (2.5pt);
\draw (3.5,4.352)-- (2.831,3.609);
\draw (4.5,2.95) node[anchor=north] {$f$};
\draw (2.5,2.95) node[anchor=north] {$f_1$};
\draw (3.25,3.78) node[anchor=north] {$t_1$};
\draw (3.75,3.78) node[anchor=north] {$t_2$};
\draw (3.5,2.5) node[anchor=north] {$t_3$};
%\begin{scriptsize}
\draw [fill=black] (3.5,2.866) circle (2.5pt);
\draw[color=black] (3.75,2.85) node {$v_1$};
\draw [fill=black] (4,2) circle (2.5pt);
\draw[color=black] (4,1.8) node {$v_2$};
\draw [fill=black] (4.978,2.208) circle (2.5pt);
\draw[color=black] (5.15,2.05) node {$v_3$};
\draw [fill=black] (5.083,3.202) circle (2.5pt);
\draw[color=black] (5.15,3.35) node {$v_4$};
\draw [fill=black] (4.169,3.609) circle (2.5pt);
\draw[color=black] (4.25,3.75) node {$v_5$};
%\end{scriptsize}
\end{tikzpicture}
}

\subfigure[\label{2needyCase1} Claim~\ref{noadjneedyvert}, Case 1.]{
\begin{tikzpicture}[scale=0.66,line cap=round,line join=round,>=triangle 45,x=2.0cm,y=2.0cm]
%\clip(1.36,1.62) rectangle (5.4,4.42); 
% cycle
\draw (3.5,2.866)-- (4,2.);
\draw (4,2.)-- (4.978,2.208);
\draw (4.978,2.208)-- (5.083,3.202);
\draw (4.169,3.609)-- (5.083,3.202);
\draw (4.169,3.609)-- (3.5,2.866);
\draw (4.3,2.9) node[anchor=north] {$f$};
% vertices on cycle
\draw [fill=black] (3.5,2.866) circle (2.5pt);
\draw [fill=black] (4,2) circle (2.5pt);
\draw [fill=black] (4.978,2.208) circle (2.5pt);
\draw [fill=black] (5.083,3.202) circle (2.5pt);
\draw[color=black] (4.85,3.15) node {$v_i$};
\draw [fill=black] (4.169,3.609) circle (2.5pt);
\draw[color=black] (4.3,3.35) node {$v_{i+1}$};
% triangle between
\draw (4.169,3.609) -- (4.9,4.2) --  (5.083,3.202);
\draw [fill=black] (4.9,4.2) circle (2.5pt);
\draw[color=black] (4.8,3.95) node {$u$};
\draw[color=black] (4.7,3.59) node {$t$};
% diamond on v_{i+1}
\draw (4.169,3.609) -- (3.25,3.9) --  (3.5,2.866);
\draw (4.169,3.609) -- (3.8,4.5) --  (3.25,3.9);
\draw [fill=black] (3.25,3.9) circle (2.5pt);
\draw [fill=black] (3.8,4.5) circle (2.5pt);
%\draw[style=dashed] (3.8,4.5) to[out=70,in=160] (4.5,4.9) to[out=-20,in=100] (4.9,4.2);
\draw (3.8,4.5) -- (4.2,5.0) -- (4.8,4.85) -- (4.9,4.2);
\draw[fill=black]  (4.2,5.0) circle (2.5pt);
\draw[fill=black] (4.8,4.85)  circle (2.5pt);
\draw[fill=white]  (4.2,5.0) circle (2pt);
\draw[fill=white] (4.8,4.85)  circle (2pt);
\draw[color=black] (4.4,4.3) node {$g_{i+1}$};
% diamond on v_{i}
\draw (5.083,3.202) -- (5.8,2.7) --  (4.978,2.208);
\draw (5.083,3.202) -- (5.85,3.7) --   (5.8,2.7);
\draw [fill=black]  (5.8,2.7) circle (2.5pt);
\draw [fill=black] (5.85,3.7) circle (2.5pt);
%\draw[style=dashed] (5.7,3.7)  to[out=80,in=-30] (5.6,4.5) to[out=150,in=60] (4.9,4.2);
\draw (5.85,3.7)  -- (5.95,4.35) -- (5.35,4.65) -- (4.9,4.2);
\draw [fill=black] (5.95,4.35) circle (2.5pt);
\draw [fill=black] (5.35,4.65) circle (2.5pt);
\draw [fill=white] (5.95,4.35) circle (2pt);
\draw [fill=white] (5.35,4.65) circle (2pt);
\draw[color=black] (5.25,4.80) node {$a$};
\draw[color=black] (5.4,4) node {$g_i$};
\end{tikzpicture}
}

\subfigure[\label{2needyCase2} Claim~\ref{noadjneedyvert}, Case 2.]{
\begin{tikzpicture}[scale=0.66,line cap=round,line join=round,>=triangle 45,x=2.0cm,y=2.0cm]
%\clip(1.36,1.62) rectangle (5.4,4.42); 
% cycle
\draw (3.5,2.866)-- (4,2.);
\draw (4,2.)-- (4.978,2.208);
\draw (4.978,2.208)-- (5.083,3.202);
\draw (4.169,3.609)-- (5.083,3.202);
\draw (4.169,3.609)-- (3.5,2.866);
\draw (4.3,2.9) node[anchor=north] {$f$};
% vertices on cycle
\draw [fill=black] (3.5,2.866) circle (2.5pt);
\draw [fill=black] (4,2) circle (2.5pt);
\draw [fill=black] (4.978,2.208) circle (2.5pt);
\draw [fill=black] (5.083,3.202) circle (2.5pt);
\draw[color=black] (4.85,3.15) node {$v_i$};
\draw [fill=black] (4.169,3.609) circle (2.5pt);
\draw[color=black] (4.3,3.35) node {$v_{i+1}$};
% triangle between
\draw (4.169,3.609) -- (4.9,4.2) --  (5.083,3.202);
\draw [fill=black] (4.9,4.2) circle (2.5pt);
\draw[color=black] (4.7,3.59) node {$t$};
\draw[color=black] (4.8,3.95) node {$u$};
\draw [fill=black] (5.7,3.8) circle (2.5pt);
\draw (4.9,4.2)--(5.7,3.8)--(5.083,3.202);
\draw[color=black] (5.8,3.95) node {$w$};
% diamond on v_{i+1}
\draw (4.169,3.609) -- (3.25,3.9) --  (3.5,2.866);
\draw (4.169,3.609) -- (3.8,4.5) --  (3.25,3.9);
\draw [fill=black] (3.25,3.9) circle (2.5pt);
\draw [fill=black] (3.8,4.5) circle (2.5pt);
%\draw[style=dashed] (3.8,4.5) to[out=70,in=160] (4.5,4.9) to[out=-20,in=100] (4.9,4.2);
\draw (3.8,4.5) -- (4.2,5.0) -- (4.8,4.85) -- (4.9,4.2);
\draw[fill=black]  (4.2,5.0) circle (2.5pt);
\draw[fill=black] (4.8,4.85)  circle (2.5pt);
\draw[fill=white]  (4.2,5.0) circle (2pt);
\draw[fill=white] (4.8,4.85)  circle (2pt);
\draw[color=black] (4.05,5.1) node {$b$};
\draw[color=black] (4.4,4.3) node {$g_{i+1}$};
% diamond on v_{i}
\draw (5.083,3.202) -- (5.8,2.7) --  (4.978,2.208);
\draw [fill=black]  (5.8,2.7) circle (2.5pt);
%\draw[style=dashed] (5.8,2.7) to[out=0,in=-90] (6.3,3.2) to[out=90,in=00] (5.7,3.8);
\draw (5.8,2.7) -- (6.3,2.95) -- (6.25,3.6) -- (5.7,3.8);
\draw [fill=black]  (6.3,2.95) circle (2.5pt);
\draw [fill=black]  (6.25,3.6) circle (2.5pt);
\draw [fill=white]  (6.3,2.95) circle (2pt);
\draw [fill=white]  (6.25,3.6) circle (2pt);
\draw[color=black] (5.7,3.25) node {$g_i$};
%\end{scriptsize}
\end{tikzpicture}
}
}
\caption{\label{last5cyclefig} Special cases for a 5-face $f$ with $\nu_4(f) < 0$.}
\end{figure}
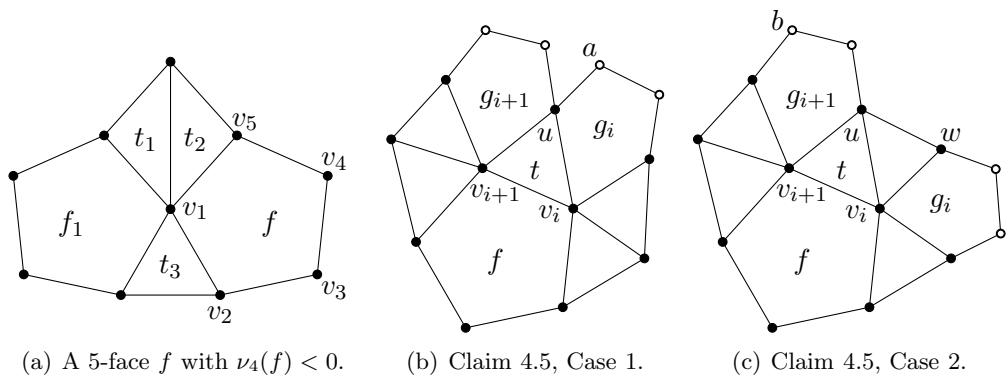

Consider the Figure~\ref{2needy}, where $f$ is a 5-face with $\nu_4(f) < 0$, $f$ is incident to vertices $v_1,\dots,v_5$, $v_1$ is a needy 5-vertex, and $f_1$ is the needy 5-face incident to $v_1$.
Let $t_1$ and $t_2$ be the adjacent pair of 3-faces incident to $v_1$ with $t_1$ adjacent to $f_1$ and $t_2$ adjacent to $f$; let $t_3$ be the other 3-face incident to $v_1$.
We make two basic claims about this arrangement.

\begin{claim}
The vertex $v_2$ adjacent to $v_1$ and incident to $t_3$ is a $5^+$-vertex.
\label{adj5vert}
\end{claim}

\begin{proof}
If $v_2$ is a 4-vertex, then the vertices incident to $f_1$ and $t_3$ form \ref{d2}, a contradiction.
\end{proof}

\begin{claim}
If $v_i$ and $v_{i+1}$ are consecutive vertices on the border of $f$, then at most one of $v_i$ and $v_{i+1}$ is needy.
\label{noadjneedyvert}
\end{claim}

\begin{proof}
Suppose that two consecutive vertices $v_i$ and $v_{i+1}$ are needy 5-vertices.  
Let $g_i$ and $g_{i+1}$ be the needy 5-faces incident to $v_i$ and $v_{i+1}$, respectively.
Since both $v_i$ and $v_{i+1}$ have three incident 3-faces, $f$ is adjacent to a 3-face $t$ across the edge $v_i v_{i+1}$.  
Let $u$ be the third vertex incident to $t$ and consider two cases.

\begin{mycases}
\mycase{$t$ is not in a diamond (Figure~\ref{2needyCase1}).}
Since $g_i$ is needy, the vertex $a$ adjacent to $u$ and incident to $g_i$ (with $a \neq v_i$) is a 4-vertex and is incident to a 3-face $t_i$ such that $t_i$ is adjacent to $g_i$.
The vertices incident to $g_i$, $g_{i+1}$, $t$, and $t_i$ form one of \ref{d7} or \ref{d8}, a contradiction.

\mycase{$t$ is in a diamond (Figure~\ref{2needyCase2}).} 
Let $w$ be the fourth vertex in the diamond and assume, without loss of generality, that $v_i$ is adjacent to $w$.  
Let $b$ be the vertex incident to $g_{i+1}$ that is not adjacent to $u$ or $v_{i+1}$ along the boundary of $g_{i+1}$; since $g_{i+1}$ is needy, there is a 3-face $t_{i+1}$ incident to $b$ and adjacent to $g_{i+1}$.
The vertices $v_{i}$ and $w$ and those incident to $g_{i+1}$ and $t_{i+1}$ form one of \ref{d5} or \ref{d6}, a contradiction.\qedhere
\end{mycases}
\end{proof}

By Claim~\ref{noadjneedyvert}, $f$ is incident to at most two needy vertices, and by Claim~\ref{adj5vert}, $v_2$ is non-needy.
If $f$ is incident to exactly one needy 5-vertex, then $v_3,v_4,$ and $v_5$ are 4-vertices since $\mu_2(f) = 0$, but then the vertices incident to $f$ and $f_1$ form \ref{d9}, a contradiction.

Therefore, $f$ is incident to two needy vertices, and since $v_2$ is a $5^+$-vertex, $f$ is incident to exactly two 4-vertices.
Each of these receives charge 1, so $\nu_4 (f) = -\frac12$.
By Claim~\ref{noadjneedyvert}, the needy vertices incident to $f$ consist of $v_1$ and exactly one of $v_3$ or $v_4$.
The needy 5-vertex $v_i$ other than $v_1$ is also incident to three 3-faces $t_4, t_5$, and $t_6$, where $t_4$ and $t_5$ form a diamond with $t_4$ adjacent to $f$.
By Claim~\ref{adj5vert}, the vertex adjacent to $v_i$ and incident to both $f$ and $t_6$ is  a non-needy $5^+$-vertex.
The only non-needy $5^+$-vertex incident to $f$ is $v_2$, and hence $v_3$ is a needy 5-vertex and $t_4$ is incident to $v_4$.
If $v_2$ is a $6^+$-vertex, then $\nu_4(f) \geq 0$.
Therefore, there is a unique arrangement of needy vertices, 4-vertices, and a 5-vertex about a 5-face $f$ with $\nu_4(f) < 0$ (Figure~\ref{vnonneedyf2hneedy}).
For $i \in \{1,3\}$, let $f_i$ be the needy $5$-face incident to the needy 5-vertex $v_i$.

\begin{figure}[h]
\centering
\begin{tikzpicture}[scale=0.66,line cap=round,line join=round,>=triangle 45,x=2.0cm,y=2.0cm]
%\clip(1.38,1.7) rectangle (6.6,5.06);
\draw [fill=black] (4,3) circle (2.5pt);
\draw [fill=black] (3.5,4.539) circle (2.5pt);
\draw [fill=black] (4.5,4.539) circle (2.5pt);
\draw [fill=black] (4.809,3.588) circle (2.5pt);
\draw [fill=black] (3.191,3.588) circle (2.5pt);
\draw [fill=black] (3.086,2.593) circle (2.5pt);
\draw [fill=black] (1.608,3.251) circle (2.5pt);
\draw [fill=black] (2.277,3.995) circle (2.5pt);
\draw [fill=black] (2.108,2.385) circle (2.5pt);
\draw [fill=black] (3.086,2.593) circle (2.5pt);
\draw [fill=black] (4.914,2.593) circle (2.5pt);
\draw [fill=black] (5.723,3.995) circle (2.5pt);
\draw [fill=black]  (6.392,3.251) circle (2.5pt);
\draw [fill=black] (5.414,4.946) circle (2.5pt);
\draw [fill=black] (5.892,2.385) circle (2.5pt);
\draw [fill=black] (2.586,4.946) circle (2.5pt);
\draw (3.5,4.539)-- (4.5,4.539);
\draw (4.5,4.539)-- (4.809,3.588);
\draw[] (4.5,4.539)--(4.5,4.9939);
\draw[] (3.5,4.539)--(3.5,4.9939);
\draw (4.809,3.588)-- (4.,3.);
\draw (4.,3.)-- (3.191,3.588);
\draw (3.191,3.588)-- (3.5,4.539);
\draw (4.,3.)-- (3.086,2.593);
\draw (3.086,2.593)-- (3.191,3.588);
\draw (3.191,3.588)-- (2.277,3.995);
\draw (2.277,3.995)-- (1.608,3.251);
\draw (1.608,3.251)-- (2.108,2.385);
\draw (2.108,2.385)-- (3.086,2.593);
\draw (4.,3.)-- (4.914,2.593);
\draw (4.914,2.593)-- (4.809,3.588);
\draw (4.809,3.588)-- (5.723,3.995);
\draw (5.723,3.995)-- (6.392,3.251);
\draw (6.392,3.251)-- (5.892,2.385);
\draw (5.892,2.385)-- (4.914,2.593);
\draw (2.277,3.995)-- (2.586,4.946);
\draw (2.586,4.946)-- (3.5,4.539);
\draw (4.5,4.539)-- (5.414,4.946);
\draw (5.414,4.946)-- (5.723,3.995);
\draw (3.191,3.588)-- (2.586,4.946);
\draw (4.809,3.588)-- (5.414,4.946);
\draw [] (3.086,2.593)-- (3.086,1.843);
\draw [] (4.,3.)-- (4.,2.25);
\draw [] (4.914,2.593)-- (4.914,1.843);
\draw (2.6,4.2) node {$t_1$};
\draw (3.15,4.4) node {$t_2$};
\draw (4.9,4.4) node {$t_4$};
\draw (5.35,4.2) node {$t_5$};
\draw (3.4,3.05) node {$t_3$};
\draw (4.6,3.05) node {$t_6$};
%\draw (3.55,2.4) node {$g_1$};
\draw (5.55,3.2) node {$f_3$};
\draw (2.45,3.2) node {$f_1$};
\draw (4,3.9) node {$f$};
%\draw (4.45,2.4) node {$g_2$};
\draw[color=black] (4,3.2) node {$v_2$};
\draw[color=black] (3.45,3.638) node {$v_1$};
\draw[color=black] (4.55,3.638) node {$v_3$};
\draw[color=black] (4.35,4.339) node {$v_4$};
\draw[color=black] (3.65,4.339) node {$v_5$};
\end{tikzpicture}
\caption{\label{vnonneedyf2hneedy} A non-needy 5-vertex $v_2$ incident to a non-needy 5-face $f$ with $\nu_4(f) < 0$.}
\end{figure}
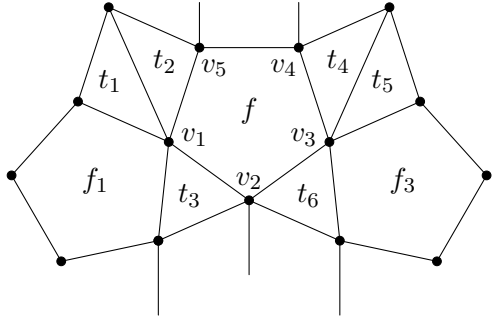

The vertices incident to $f$, $f_1$, $f_3$, $t_3$, and $t_6$ form \ref{bigneedy}, so this arrangement does not appear within $G$; hence $\nu_4(f) \geq 0$ for all 5-faces $f$.
Therefore, every vertex and face has nonnegative charge after (R4), contradicting the negative initial charge sum.
Thus, a minimal counterexample does not exist and every plane graph with no chorded 5-cycle is $(4,2)$-choosable.
\end{proof}

%%%%%%%%%%%%%%%%%%%%%%%%%%%%%%%%%%%%%%%%%%%%%%%%%%%%%%%%%%%%
%%%%%%%%%%%%%%%%%%%%%%%%%%%%%%%%%%%%%%%%%%%%%%%%%%%%%%%%%%%%
%%%%%%%%%%%%%%%%%%%%%%%%%%%%%%%%%%%%%%%%%%%%%%%%%%%%%%%%%%%%
%%%%%%%%%%%%%%% No Chorded 6-cycle   %%%%%%%%%%%%%%%%%%%%%%%%%%%%%%%%%%
%%%%%%%%%%%%%%%%%%%%%%%%%%%%%%%%%%%%%%%%%%%%%%%%%%%%%%%%%%%%
%%%%%%%%%%%%%%%%%%%%%%%%%%%%%%%%%%%%%%%%%%%%%%%%%%%%%%%%%%%%
%\clearpage
\section{No Chorded 6-Cycle}\label{sec:6cycles}

In this section we show the case of forbidding chorded 6-cycles from Theorem~\ref{thm:42choosableconditions}.
The case of forbidding doubly-chorded 6- and 7-cycles follows from a very similar argument.
We give the full proof for no chorded 6-cycles and describe the differences for the proof when we forbid doubly-chorded 6- and 7-cycles.

\begin{theorem}\label{thm:cc6}
% We need to have plane graphs if we talk later about faces
If $G$ is a plane graph not containing a chorded 6-cycle, then $G$ is $(4,2)$-choosable.
\end{theorem}

\begin{proof}
Let $G$ be a counterexample minimizing $n(G)$ among all plane graphs avoiding chorded 6-cycles with a $(4,2)$-list assignment $L$ such that $G$ is not $L$-choosable.
%Suppose there exists a counterexample.
%Select a counterexample $G$ by minimizing $n(G)$ among all planar graphs avoiding chorded 6-cycles with a $(4,2)$-list assignment $L$ such that $G$ is not $L$-choosable.
Observe that $n(G) \geq 5$; in fact, $\delta(G) \geq 4$.
Since $G$ is a minimal counterexample, $G$ does not contain any of the reducible configurations.
Specifically, we use the fact that $G$ avoids \ref{diamond1} and \ref{diamond2} (see Figure~\ref{fig:configurations}).

For each $v \in V(G)$ and $f \in F(G)$ define initial charge $\mu(v) = d(v) - 4$ and $\nu(f) = \ell(f) - 4$.
By Euler's Formula, the initial charge sum is $-8$.
Since $\delta(G) \geq 4$, the only elements of negative charge are 3-faces.
Since a chorded 6-cycle is forbidden and $\delta(G)\geq 4$, the clusters (see Figure~\ref{fig:clusters}) are triangles (K3), diamonds (K4), 3-fans (K5a), 4-wheels (K5b), and 4-fans with end vertices identified (K5c).
Specifically note that the 4-fan (K6b) contains a chorded 6-cycle, so at most three 3-faces in a cluster share a common vertex, unless they form a 4-wheel (K5b) and the common vertex is the 4-vertex in the center of the wheel.

%If $v$ is a vertex and $f$ is an incident diamond or fan, then let $t(v,f)$ be the number of triangles of $f$ that include the vertex $v$.
Apply the following discharging rules, as shown in Figure~\ref{fig-rules6}.

\begin{itemize}
\item[(R1)] If $f$ is a 3-face and $e$ is an incident edge, then let $g$ be the face adjacent to $f$ across $e$.

\begin{itemize}
\item[(R1a)] If $g$ is a $5^+$-face, then $f$ pulls charge $\frac{1}{3}$ from $g$ ``through'' the edge $e$.

\item[(R1b)] If $g$ is a 4-face, then let $e_1$, $e_2$, and $e_3$ be the other edges incident to $g$. 
For each $i \in \{1,2,3\}$, let $h_i$ be the face adjacent to $g$ across $e_i$. For each $i \in \{1,2,3\}$, the face $f$ pulls charge $\frac{1}{9}$ from the face $h_i$ ``through'' the edges $e$ and $e_i$.
\end{itemize}

\item[(R2)] Let $v$ be a $5^+$-vertex, and let $f$ be an incident 3-face.

\begin{itemize}
\item[(R2a)] If $v$ is a $5$-vertex, then $v$ sends charge $\frac{1}{3}$ to $f$.

\item[(R2b)] If $v$ is a $6^+$-vertex, then $v$ sends charge $\frac{4}{9}$ to $f$.
\end{itemize}
% At some point it is needed to redistribute charge inside clusters
\item[(R3)] If $X$ is a cluster, then every 3-face in $X$ is assigned the average charge of all 3-faces in $X$. 
\end{itemize}

\begin{figure}[htp]
\def\sixcyclerulescale{0.75}
\begin{center}
\begin{tabular}[h]{VVVV}
\begin{tikzpicture}[scale=\sixcyclerulescale,vtx/.style={shape=coordinate}]
	% tikz here
	\node[vtx] (a) at (0,0) {};
	\node[vtx] (b) at (180:3) {};
	\node[vtx] (c) at ($(b) !1! 60:(a)$) {};
	\foreach \p in {a,b,c}
	{
		\fill (\p) circle (2pt);
	}
	\draw (a) -- (b) -- (c) -- (a) -- ($(a) !0.4! 180:(b)$);
	\draw (c) -- ($(c) !0.4! 165:(b)$);
	\begin{scope}[>=latex,very thick,densely dashed]
		\def\mytheta{25}
		\draw[->] ($(a) !1/(2*cos(\mytheta))! -\mytheta:(c)$) -- ($(a) !1/(2*cos(\mytheta))! \mytheta:(c)$);
		\node at ($(a) !0.69! -17:(c)$) {$\frac{1}{3}$};
	\end{scope}
	\node at ($(b) !0.4! 32:(a)$) {$f$};
	\node at ($(b) !1.3! 30:(a)$) {$g$};
	\node at ($(b) !0.8! 15:(a)$) {$e$};
\end{tikzpicture}
&
\begin{tikzpicture}[scale=\sixcyclerulescale,vtx/.style={shape=coordinate}]
	% tikz here
	\node[vtx] (a) at (0,0) {};
	\node[vtx] (b) at (180:3) {};
	\node[vtx] (c) at ($(b) !1! 90:(a)$) {};
	\node[vtx] (d) at ($(c) !1! 90:(b)$) {};
	\node[vtx] (e) at ($(c) !1/(2*cos(45))! 45:(d)$) {};
	\foreach \p in {a,b,c,d,e}
	{
		\fill (\p) circle (2pt);
	}
	\draw (a) -- (b) -- (c) -- (d) -- (e) -- (c) -- (d) -- (a);
	\draw (a) -- +(-30:1) -- (a) -- +(-60:1);
	\draw (b) -- +(-120:1) -- (b) -- +(210:1);
	\draw (c) -- +(150:1);
	\draw (d) -- +(30:1);
	\begin{scope}[>=latex,very thick,densely dashed]
		\def\mytheta{25}
		\draw[->] ($(b) !1/(2*cos(\mytheta))! -\mytheta:(a)$) -- ($(c) !1/(2*cos(\mytheta-10))! \mytheta-10:(d)$);
		\node at ($(b) !0.6! -15:(a)$) {$\frac{1}{9}$};
		\draw[->] ($(c) !1/(2*cos(\mytheta))! -\mytheta:(b)$) .. controls ($(c) !1/(2*cos(\mytheta))! \mytheta:(b)$) .. ($(c) !1/(4*cos(\mytheta))! \mytheta:(d)$);
		\node at ($(c) !0.35! -17:(b)$) {$\frac{1}{9}$};
		\draw[->] ($(a) !1/(2*cos(\mytheta))! -\mytheta:(d)$) .. controls ($(a) !1/(2*cos(\mytheta))! \mytheta:(d)$) .. ($(d) !1/(4*cos(\mytheta))! -\mytheta:(c)$);
		\node at ($(a) !0.7! -12:(d)$) {$\frac{1}{9}$};
	\end{scope}
	\node at ($(e) !0.3! -45:(d)$) {$f$};
	\node at ($(b) !0.8! 35:(a)$) {$g$};
	\node at ($(d) !0.4! 11:(c)$) {$e$};
	\node at ($(a) !0.3! -45:(d)$) {$h_1$};
	\node at ($(a) !0.4! 12:(d)$) {$e_1$};
	\node at ($(b) !0.3! -45:(a)$) {$h_2$};
	\node at ($(b) !0.37! 12:(a)$) {$e_2$};
	\node at ($(b) !0.3! 45:(c)$) {$h_3$};
	\node at ($(b) !0.4! -14:(c)$) {$e_3$};
\end{tikzpicture}
&
\begin{tikzpicture}[scale=\sixcyclerulescale,vtx/.style={shape=coordinate}]
	% tikz here
	\node[vtx] (a) at (0,0) {};
	\node[vtx] (b) at (180:3) {};
	\node[vtx, label=left:$v$] (v) at ($(b) !1! 60:(a)$) {};
	\foreach \p in {a,b,v}
	{
		\fill (\p) circle (2pt);
	}
	\draw (a) -- (b) -- (v) -- (a);
	\draw (v) -- +(45:1) -- (v) -- +(90:1) -- (v) -- +(135:1);
	\begin{scope}[>=latex,very thick,densely dashed]
		\def\mytheta{30}
		\node[vtx] (tip) at ($(v) !0.5! ($(b) !0.5! (a)$)$) {};
		\draw[->] (v) -- (tip);
		\node at ($(v) !0.9! 10:(tip)$) {$\frac{1}{3}$};
	\end{scope}
	\node at ($(v) !1.5! (tip)$) {$f$};
\end{tikzpicture}
&
\begin{tikzpicture}[scale=\sixcyclerulescale,vtx/.style={shape=coordinate}]
	% tikz here
	\node[vtx] (a) at (0,0) {};
	\node[vtx] (b) at (180:3) {};
	\node[vtx, label=left:$v$] (v) at ($(b) !1! 60:(a)$) {};
	\foreach \p in {a,b,v}
	{
		\fill (\p) circle (2pt);
	}
	\draw (a) -- (b) -- (v) -- (a);
	\draw (v) -- +(45:1) -- (v) -- +(65:1) -- (v) -- +(115:1) -- (v) -- +(+(135:1);
	\draw[dashed] (v) -- +(90:1);
	\begin{scope}[>=latex,very thick,densely dashed]
		\def\mytheta{30}
		\node[vtx] (tip) at ($(v) !0.5! ($(b) !0.5! (a)$)$) {};
		\draw[->] (v) -- (tip);
		\node at ($(v) !0.9! 10:(tip)$) {$\frac{4}{9}$};
	\end{scope}
	\node at ($(v) !1.5! (tip)$) {$f$};
\end{tikzpicture}
\\
(R1a) &
(R1b) &
(R2a) &
(R2b) \\
\end{tabular}
\end{center}
\caption{Discharging rules in the proof of Theorem~\ref{thm:cc6}.}\label{fig-rules6}
\end{figure}
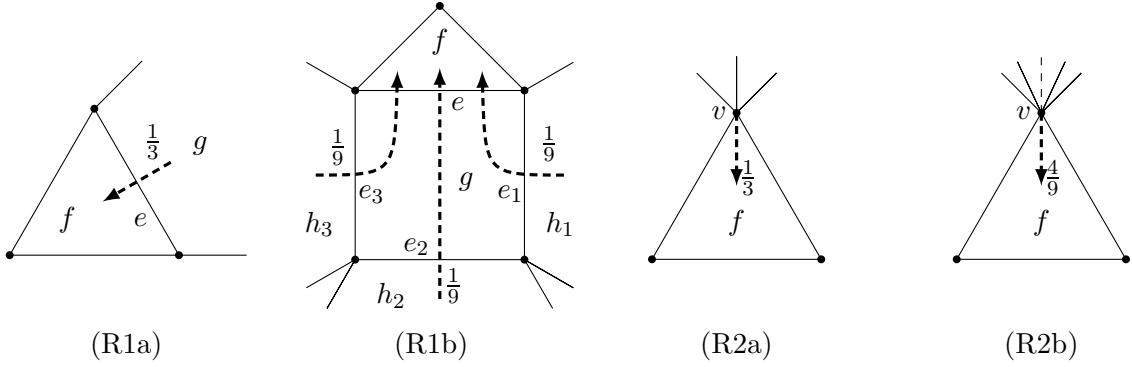

Notice that the rules preserve the sum of the charges.
Let $\mu_i(v)$ and $\nu_i(f)$ denote the charge on a vertex $v$ or a face $f$ after rule (R$i$).
We claim that $\mu_3(v) \geq 0$ for every vertex $v$ and $\nu_3(f) \geq 0$ for every face $f$; since the total charge sum is preserved by the discharging rules, this contradicts the negative charge sum from the initial charge values.

Let $v$ be a vertex.
If $v$ is a 4-vertex, then $v$ is not involved in any rule, so the resulting charge is $0$.
If $v$ is a 6$^+$-vertex, then by (R2b) $v$ loses charge $\frac{4}{9}$ to each incident 3-face.
Since $G$ avoids chorded 6-cycles, $v$ is incident to at most $\lfloor \frac{3}{4}d(v)\rfloor$ 3-faces.
Thus $\mu_3(v)$ satisfies
\[
\mu_3(v) \geq d(v) - 4 - \frac{4}{9}\left\lfloor \frac{3}{4}d(v)\right\rfloor \geq d(v) - 4 - \frac{4}{9}\cdot \frac{3}{4}d(v) = \frac{2}{3}d(v)-4\geq 0.
\]
If $v$ is a 5-vertex, then by (R2a) $v$ loses charge $\frac{1}{3}$ to each incident 3-face. Since $G$ avoids chorded 6-cycles, $v$ is incident to at most three 3-faces, so
\[
\mu_3(v) \geq d(v) - 4 - \frac{1}{3} \cdot 3= d(v) - 5  = 0.
\]
Therefore, $\mu_3(v) \geq 0$ for every vertex $v$.

Let $f$ be a face.
%If $f$ is a 4-face, then (R1b) does not pull charge from $f.$ This is because it would require $f$ to be adjacent to a 4-face $g,$ which is adjacent to a 3-face. Together $f$ and $g$ would form a chorded 6-cycle.
Since $4$-faces are not adjacent to $4$-faces, (R1b) does not affect the charge value on 4-faces.
Thus, $\nu_3(f) = 0$ for every 4-face $f$.

If $f$ is a $6^+$-face, then $f$ loses charge at most $\frac{1}{3}$ through each edge by (R1a) or (R1b), so
\[
	\nu_3(f) \geq \ell(f) - 4 - \frac{1}{3}\ell(f) = \frac{2}{3}\ell(f) - 4 \geq 0.
\]
Therefore, $\nu_3(f) \geq 0$ for every $6^+$-face $f$.

Let $f$ be a 5-face.
Since $G$ contains no chorded 6-cycles, $f$ is not adjacent to a 3-face.
Therefore, $f$ loses no charge by (R1a), but could lose charge using (R1b), so
\[
	\nu_3(f) \geq \ell(f) - 4 - \frac{1}{9}\ell(f) = \frac{8}{9} \ell(f)-4\geq 0.
\]
Therefore, $\nu_3(f) \geq 0$ if $f$ is a $5$-face.
All objects that start with nonnegative charge have nonnegative charge after the discharging process.
It remains to show that each cluster of 3-faces receives enough charge to result in a nonnegative charge sum.
%Then the averaging process of (R3) results in nonnegative charge on all 3-faces.

\begin{mycases}

\mycase{(K3)}
Let $f$ be an isolated 3-face.
The three adjacent faces $g_1$, $g_2$, and $g_3$ are all $4^+$-faces.
By (R1a) or (R1b), $f$ receives charge $\frac{1}{3}$ through each incident edge, so $\nu_3(f) = -1 + 3\cdot\frac{1}{3} = 0$.

\mycase{(K4)}
Let $f_1$ and $f_2$ be 3-faces in a diamond cluster (K4).
Then $f_1$ is adjacent to two $4^+$-faces $g_1$ and $g_2$, and $f_2$ is adjacent to two $4^+$-faces $h_1$ and $h_2$.
By (R1a) or (R1b), the cluster receives charge $\frac{1}{3}$ through each of the four edges on the boundary of the diamond.
Since $\nu(f_1)+\nu(f_2) = -2$, the charge value on the diamond after rule (R1) is $-\frac{2}{3}$.
Since $G$ contains no \ref{diamond1}, there is a $5^+$-vertex $v$ incident to both $f_1$ and $f_2$.
If $v$ is a $5$-vertex, then by (R2a), $f_1$ and $f_2$ each receive charge $\frac{1}{3}$, and the resulting charge on the diamond is zero.
If $v$ is a $6^+$-vertex, then by (R2b), $f_1$ and $f_2$ each receive charge $\frac{4}{9}$, and the resulting charge on the diamond is positive.

\mycase{(K5a)}
Let $f_1$, $f_2$, and $f_3$ be 3-faces in a 3-fan cluster (K5a), where $f_2$ is adjacent to both $f_1$ and $f_3$.
The initial charge on this cluster is $-3$.
There are five edges on the boundary of this cluster, so by (R1) the cluster receives charge $\frac{5}{3}$, resulting in charge $-\frac{4}{3}$ after (R1).
Note that the face $f_2$ is adjacent to both $f_1$ and $f_3$.
Since $G$ contains no \ref{diamond1}, there exists a $5^+$-vertex $v$ incident to both $f_1$ and $f_2$,
and there exists a $5^+$-vertex $u$ incident to both $f_2$ and $f_3$.
If $v \neq u$, then by (R2) $v$ sends charge at least $\frac{1}{3}$ to each of $f_1$ and $f_2$ and $u$ sends charge at least $\frac{1}{3}$ to each of $f_2$ and $f_3$, resulting in a nonnegative charge on the 3-fan.
If $v = u$ and $v$ is a $6^+$-vertex, then by (R2b) $v$ sends charge $\frac{4}{9}$ to each face $f_1$, $f_2$, and $f_3$, resulting in a nonnegative charge on the 3-fan.
Otherwise, suppose that $v=u$ and $v$ is a $5$-vertex.
Since $G$ contains no \ref{diamond2}, there exists another $5^+$-vertex $w$ incident to at least one of $f_1$ and $f_2$.
By (R2a) $v$ sends charge $\frac{1}{3}$ to each of $f_1$, $f_2$, and $f_3$, and by (R2) $w$ sends charge at least $\frac{1}{3}$ to at least one of $f_1$ and $f_2$, resulting in a nonnegative charge on the 3-fan.

\mycase{(K5b)}
Let $f_1$, $f_2$, $f_3$, and $f_4$ be 3-faces in a 4-wheel (K5b).
The initial charge on this cluster is $-4$.
There are four edges on the boundary of this cluster, so by (R1) the cluster receives charge $\frac{4}{3}$, resulting in charge $-\frac{8}{3}$ after (R1).
Let $v$ be the 4-vertex incident to all four 3-faces.
Let $u_1$, $u_2$, $u_3$, and $u_4$ be the vertices adjacent to $v$, ordered cyclically such that $vu_iu_{i+1}$ is the boundary of the 3-face $f_i$ for $i\in\{1,2,3\}$ and $vu_4u_1$ is the boundary of $f_4$.
Since $G$ contains no \ref{diamond1} and $d(v) = 4$, each $u_i$ is a $5^+$-vertex.
By (R2), each $u_i$ sends charge at least $\frac{2}{3}$ to the cluster, resulting in a nonnegative total charge.

\mycase{(K5c)}
Let $f_1$, $f_2$, $f_3$, and $f_4$ be 3-faces in a 4-strip with identified vertices as in (K5c).  
The initial charge on this cluster is $-4$.
Let $v$, $u_1$, $u_2$, $u_3$, and $u_4$ be the vertices in the 4-strip, where $v$ is incident to only $f_1$ and $f_4$, $u_1$ is incident to only $f_1$ and $f_2$, $u_2$ is incident to $f_2$, $f_3$, and $f_4$, $u_3$ is incident to $f_1$, $f_2$, and $f_3$, and $u_4$ is incident to only $f_3$ and $f_4$.
There are six edges on the boundary of this cluster, so by (R1) the cluster receives charge $\frac{6}{3}$, resulting in charge $-\frac{6}{3}=-2$ after (R1).

Since $f_2$ and $f_3$ form a diamond, and $G$ contains no \ref{diamond1}, one of $u_2$ and $u_3$ is a $5^+$-vertex. 
Without loss of generality, assume $u_3$ is a $5^+$-vertex.
Since $f_3$ and $f_4$ form a diamond, and $G$ contains no \ref{diamond1}, one of $u_2$ and $u_4$ is a $5^+$-vertex.
If $u_2$ is a  $5^+$-vertex, then by (R2), the cluster receives charge at least $\frac{3}{3}+\frac{3}{3}$ from $u_2$ and $u_3$, which results in nonnegative total charge. Otherwise, $u_2$ is a 4-vertex and $u_4$ is $5^+$-vertex.
If $u_3$ is a $6^+$-vertex,  then by (R2), the cluster receives charge at least $\frac{4}{3}+\frac{2}{3}$ from $u_3$ and $u_4$.
If $u_3$ is a 5-vertex, then since $f_1$ and $f_2$ form a diamond and $G$ contains no \ref{diamond2}, one of $v$ and $u_1$ is a $5^+$-vertex. 
By (R2), the cluster receives charge at least $\frac{3}{3}+\frac{2}{3}+\frac{2}{3}$ from $u_3$ and $u_4$ and one of $v$ and $u_1$.
In either case, the final charge is nonnegative.

% Rewritten in the previous paragraph
%Since $f_1$ and $f_2$ form a diamond, and $G$ contains no \ref{diamond1}, one of $u_1$ and $u_3$ is a $5^+$-vertex.
%Since $f_3$ and $f_4$ form a diamond, and $G$ contains no \ref{diamond1}, one of $u_2$ and $u_4$ is a $5^+$-vertex.
%
%Suppose $u_2$ and $u_3$ are $5^+$-vertices.  Then by (R2), the cluster receives charge at least 2, which results in nonnegative total charge.
%Suppose $u_1$ and $u_4$ are $5^+$-vertices.  Then since $f_2$ and $f_3$ form a diamond, and $G$ contains no \ref{diamond1}, one of $u_2$ and $u_3$ is a $5^+$-vertex.  So by (R2), the cluster receives charge at least $\frac{7}{3}$, which results in nonnegative total charge.
%Suppose $u_3$ and $u_4$ are $5^+$-vertices. Then by (R2) the cluster receives charge at least $\frac{6}{3}$, unless $u_3$ and $u_4$ are 5-vertices and $v$, $u_1$, and $u_2$ are $4$-vertices.  However, since $f_1$ and $f_2$ form a diamond, and $G$ contains no \ref{diamond2}, we must have that one of $v$, $u_1$ and $u_2$ is a $5^+$-vertex, or $u_3$ is a $6^+$-vertex. In either case, (R2) dictates that the cluster receives charge at least $\frac{6}{3}$, which results in nonnegative total charge.
%Suppose $u_1$ and $u_2$ are $5^+$-vertices.  Then a similar argument to the case where $u_3$ and $u_4$ are $5^+$-vertices holds, so the cluster ends with nonnegative total charge.
\end{mycases}

We have verified that the total charge after discharging is nonnegative, contradicting the negative initial charge sum.
Thus, a minimal counterexample does not exist and every planar graph with no chorded 6-cycle is $(4,2)$-choosable.
\end{proof}

\begin{corollary}\label{cor:dcc67}
If $G$ is a plane graph not containing a doubly-chorded 6-cycle or a doubly-chorded 7-cycle, then $G$ is $(4,2)$-choosable.
\end{corollary}

\begin{proof}
Let $G$ be a minimal counterexample by minimizing $n(G)$.
Observe that $n(G) \geq 4$ and $\delta(G) \geq 4$.
Since $G$ contains no doubly-chorded 6-cycle, the clusters are 3-faces (K3), diamonds (K4), 3-fans (K5a), 4-wheels (K5b), and 4-fans with end vertices identified (K5c).

Use the same discharging argument as in Theorem~\ref{thm:cc6}, with the following changes:

\begin{itemize}
\item
If $f$ is a 4-face, then $f$ can be adjacent to a 4-face $g$.
However, since $G$ contains no doubly-chorded 7-cycle, $g$ cannot be adjacent to a 3-face.
Therefore, $f$ does not lose charge by rule (R1b).

\item If $f$ is a 5-face, then $f$ can be adjacent to at most one 3-face $g$, since $G$ contains no doubly-chorded 7-cycle.
By (R1a) $f$ loses charge $\frac{1}{3}$ across the edge it shares with $g$, and by (R1b) $f$ loses charge at most $\frac{1}{9}$ across the other four edges.
Thus
\[
	\nu_3(f) \geq \ell(f) - 4 - \frac{1}{3} - 4\frac{1}{9} = \frac{2}{9} \geq 0.
\]
\end{itemize}

All of the other arguments from the proof of Theorem~\ref{thm:cc6} hold, which shows that the resulting total charge is nonnegative, and hence a minimal counterexample does not exist.
\end{proof}

%%%%%%%%%%%%%%%%%%%%%%%%%%%%%%%%%%%%%%%%%%%%%%%%%%%%%%%%%%%%
%%%%%%%%%%%%%%%%%%%%%%%%%%%%%%%%%%%%%%%%%%%%%%%%%%%%%%%%%%%%
%%%%%%%%%%%%%%%%%%%%%%%%%%%%%%%%%%%%%%%%%%%%%%%%%%%%%%%%%%%%
%%%%%%%%%%%%%%% No Chorded 7-cycle   %%%%%%%%%%%%%%%%%%%%%%%%%%%%%%%%%%
%%%%%%%%%%%%%%%%%%%%%%%%%%%%%%%%%%%%%%%%%%%%%%%%%%%%%%%%%%%%
%%%%%%%%%%%%%%%%%%%%%%%%%%%%%%%%%%%%%%%%%%%%%%%%%%%%%%%%%%%%
%\clearpage
\section{No Chorded 7-Cycle}
\label{sec:nochorded7cycle}
\label{sec:7cycles}

\begin{theorem}\label{thm:cc7}
If $G$ is a plane graph not containing a chorded 7-cycle, then $G$ is $(4,2)$-choosable.
\end{theorem}

We prove the following strengthened statement:

\begin{theorem}\label{thm:cc7strengthened}
Let $G$ be a planar graph with  no chorded 7-cycle, and let $P$ be a subgraph of $G$, where $P$ is isomorphic to one of $P_1, P_2, P_3,\text{ or }K_3$, and all vertices in $V(P)$ are incident to a common face $f$.
Let $L$ be a $(4,2)$-list assignment of $G - P$ and let $c$ be a proper coloring of $P$.
There exists an extension of $c$ to a proper coloring of $G$ such that $c(v) \in L(v)$ for all $v \in V(G-P)$.
\end{theorem}

\begin{proof}
Suppose that there exists a counterexample. Select a counterexample $(G,P,L,c)$ by minimizing $n(G)-\frac{1}{4}n(P)$ among all chorded 7-cycle free plane graphs, $G$, with a subgraph $P$ isomorphic to a graph in $\{P_1,P_2,P_3,K_3\}$, a proper coloring $c$ of $P$, and a $(4,2)$-list assignment $L$ of $G-P$ such that $c$ does not extend to an $L$-coloring of $G$.  We will refer to the vertices of $P$ as \emph{precolored vertices}.

\begin{claim}\label{clm:2conn}
$G$ is 2-connected.
\end{claim}

\begin{proof}
If $G$ is disconnected, then each connected component can be colored separately.
Suppose that $G$ has a cut-vertex $v$.
Then there exist connected subgraphs $G_1$ and $G_2$ where $G = G_1 \cup G_2$ and $V(G_1) \cap V(G_2) = \{v\}$, $n(G_1) < n(G)$, and $n(G_2) < n(G)$.
We can assume without loss of generality that $G_1$ contains at least one vertex of $P$, so let $S_1$ be the subgraph of $P$ contained in $G_1$.
If $G_2$ contains at least one vertex of $P$, i.e., if $v\in S_1$, then let $S_2$ be the subgraph of $P$ contained in $G_2$; otherwise, let $S_2$ be the vertex $v$.

Since $(G,P,L,c)$ is a minimal counterexample, there is an $L$-coloring $c_1$ of $G_1$ that extends the coloring on $S_1$.
Using the color prescribed by $c_1$ on $v$, there exists an $L$-coloring $c_2$ of $G_2$ that extends the coloring on $S_2$.
The colorings $c_1$ and $c_2$ form an $L$-coloring of $G$, a contradiction.
\end{proof}

\begin{claim}\label{clm:separating}
$G$ has no separating 3-cycles.
\end{claim}

\begin{proof}
Suppose that $P' = v_1v_2v_3$ is a separating 3-cycle of $G$.
Let $G_1$ be the subgraph of $G$ given by the exterior of $P'$ along with $P'$, and let $G_2$ be the subgraph of $G$ given by the interior of $P'$ along with $P'$.
Since $P'$ is separating, $n(G_1) < n(G)$ and $n(G_2) < n(G)$.

Since the vertices in $P$ share a common face, we can assume without loss of generality that $V(P) \subseteq V(G_1)$.
Since $(G,P,L,c)$ is a minimal counterexample, there exists an $L$-coloring $c_1$ of $G_1$.
Assign the colors from $c_1$ to $P'$.
Then there exists an $L$-coloring of $G_2$ extending the colors on $P'$, and together $c_1$ and $c_2$ form an $L$-coloring of $G$, a contradiction.
\end{proof}

\begin{claim}\label{clm:precoloredtriangles}
If $v \in V(P)$ such that $V(P) \subseteq N[v]$, then the subgraph of $G$ induced by $N(v)$ is not isomorphic to any graph in $\{P_1, P_2, P_3, K_3\}$.
\end{claim}

\begin{proof}
Suppose that there exists a vertex $v \in V(P)$ where all precolored vertices are in $N[v]$ and the subgraph $G[N(v)]$ is isomorphic to a subgraph in $\{P_1,P_2,P_3,K_3\}$.
Then consider the graph $G' = G-v$ with subgraph $P' = G[N(v)]$.
Since $|N_G[v]| \leq 4$, there exists an $L$-coloring $c'$ of $G[N[v]]$.
Since $(G,P,L,c)$ is a minimal counterexample, $c'$ extends to an $L$-coloring of $G'$, which in turn extends to an $L$-coloring of $G$, a contradiction.
\end{proof}

\begin{claim}\label{clm:precoloreddegree}
If $v \in V(P)$ has $d_G(v) \leq 2$, then $d_G(v) = 2$ and $P$ is isomorphic to $P_1$, $P_2$, or $P_3$.
\end{claim}

\begin{proof}
By Claim~\ref{clm:2conn}, $d_G(v) \neq 1$.
If $d_G(v) = 2$ and $P \cong K_3$, then $G[N_G(v)]$ is isomorphic to $P_2$, contradicting Claim~\ref{clm:precoloredtriangles}.
\end{proof}

\begin{claim}\label{clm:bigsubgraph}
$P$ is isomorphic to one of $P_3$ or $K_3$.
\end{claim}

\begin{proof}
Suppose that $P$ is not isomorphic to either $P_3$ or $K_3$.
If $P$ is isomorphic to $P_1$, then the vertex of $P$ has two consecutive neighbors $u_1$ and $u_2$ not in $P$; let $U = \{u_1,u_2\}$.  
If $P$ is isomorphic to $P_2$, then some vertex $v$ in $P$ has a neighbor $u_1$ not in $P$ that shares a face with the edge in $P$; let $U = \{u_1\}$.
Let $P'$ be the subgraph isomorphic to $P_3$ or $K_3$ given by including vertices in $U$.
There exists a proper coloring $c'$ of $P'$ that extends the coloring on $P$.  But then $(G, P', L, c')$ has $n(G) - \frac{1}{4}n(P') < n(G) - \frac{1}{4}n(P)$, so there exists an $L$-coloring of $G$ that extends $c'$, a contradiction.
\end{proof}

\begin{claim}\label{clm:degreefour}
If $v \in V(G-P)$, then $d_G(v) \geq 4$.
\end{claim}

\begin{proof}
Suppose that $v \in V(G-P)$ has degree $d(v) \leq 3$.
Then $G-v$ is a planar graph with no chorded 7-cycle containing a precolored subgraph $P$ and a list assignment $L$.
Since $(G, P, L ,c)$ is a minimum counterexample, $G-v$ has an $L$-coloring.
However, $v$ has at most three neighbors and at least four colors in the list $L(v)$.
Thus, there is an extension of the $L$-coloring of $G-v$ to an $L$-coloring of $G$, a contradiction.
\end{proof}

Observe that $n(G) \geq 4$. % and in fact $\delta(G) \geq 4$.
Recall that in a configuration $(C,X,\ex)$, an $L$-coloring of $V(C)\setminus X$ extends to all of $C$.
Because of this fact, if $G$ contains a reducible configuration $(C,X,\ex)$, then there is a precolored vertex in the set $X$, or else $G-X$ has an $L$-coloring that extends to all of $G$.
Specifically, we will use the fact that $G$ avoids \ref{cycle6}, \ref{diamond1}, \ref{diamond2}, \ref{4fan}, \ref{K6hconfig1}, \ref{K6hconfig2}, and \ref{K6hconfig3}.

For each $v \in V(G)$ and $f \in F(G)$ define % $\mu(v) = d(v) - 4$ and $\nu(f) = \ell(f) - 4$.
\[
\hs \mu(v) = d(v)-4 + 2\delta(v) \quad \text{ and } \quad \nu(f) = \ell(f) - 4+ \varepsilon(f),
\]
 where $\delta(v) \in \{0,1\}$ has value 1 if and only if $v \in V(P)$, and $\varepsilon(f) \in \{0,1\}$ has value 1 if and only if the boundary of $f$ is the set of precolored vertices, $V(P)$.
By Euler's Formula, the initial charge sum is at most $-1$.
Claims~\ref{clm:precoloreddegree} and \ref{clm:degreefour} assert that the only negatively-charged objects are 3-faces.

For a vertex $v$, let $t_k(v)$ denote the number of $k$-faces incident to $v$.
Apply the following discharging rules.
Let $\mu_i(v)$ and $\nu_i(f)$ denote the charge on a vertex $v$ or a face $f$ after rule (R$i$).

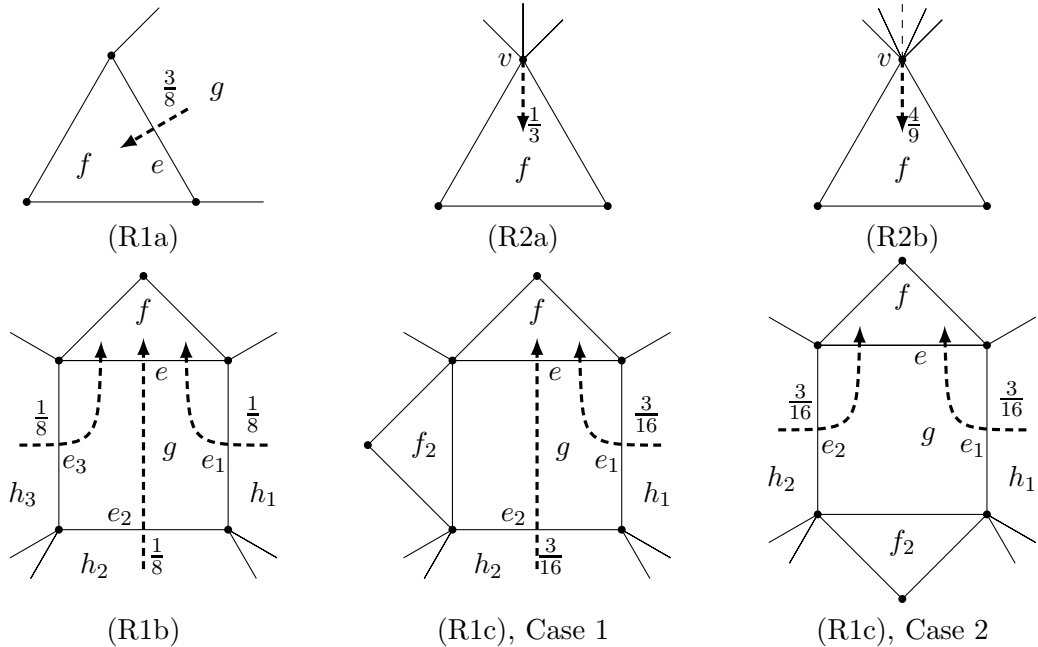
\begin{figure}[htp]
\def\sixcyclerulescale{0.75}
\begin{center}
\begin{tabular}[h]{WWW}
\begin{tikzpicture}[scale=\sixcyclerulescale,vtx/.style={shape=coordinate}]
	% tikz here
	\node[vtx] (a) at (0,0) {};
	\node[vtx] (b) at (180:3) {};
	\node[vtx] (c) at ($(b) !1! 60:(a)$) {};
	\foreach \p in {a,b,c}
	{
		\fill (\p) circle (2pt);
	}
	\draw (a) -- (b) -- (c) -- (a) -- ($(a) !0.4! 180:(b)$);
	\draw (c) -- ($(c) !0.4! 165:(b)$);
	\begin{scope}[>=latex,very thick,densely dashed]
		\def\mytheta{25}
		\draw[->] ($(a) !1/(2*cos(\mytheta))! -\mytheta:(c)$) -- ($(a) !1/(2*cos(\mytheta))! \mytheta:(c)$);
		\node at ($(a) !0.69! -17:(c)$) {$\frac{3}{8}$};
	\end{scope}
	\node at ($(b) !0.4! 32:(a)$) {$f$};
	\node at ($(b) !1.3! 30:(a)$) {$g$};
	\node at ($(b) !0.8! 15:(a)$) {$e$};
\end{tikzpicture}
&
\begin{tikzpicture}[scale=\sixcyclerulescale,vtx/.style={shape=coordinate}]
	% tikz here
	\node[vtx] (a) at (0,0) {};
	\node[vtx] (b) at (180:3) {};
	\node[vtx, label=left:$v$] (v) at ($(b) !1! 60:(a)$) {};
	\foreach \p in {a,b,v}
	{
		\fill (\p) circle (2pt);
	}
	\draw (a) -- (b) -- (v) -- (a);
	\draw (v) -- +(45:1) -- (v) -- +(90:1) -- (v) -- +(135:1);
	\begin{scope}[>=latex,very thick,densely dashed]
		\def\mytheta{30}
		\node[vtx] (tip) at ($(v) !0.5! ($(b) !0.5! (a)$)$) {};
		\draw[->] (v) -- (tip);
		\node at ($(v) !0.9! 10:(tip)$) {$\frac{1}{3}$};
	\end{scope}
	\node at ($(v) !1.5! (tip)$) {$f$};
\end{tikzpicture}
&
\begin{tikzpicture}[scale=\sixcyclerulescale,vtx/.style={shape=coordinate}]
	% tikz here
	\node[vtx] (a) at (0,0) {};
	\node[vtx] (b) at (180:3) {};
	\node[vtx, label=left:$v$] (v) at ($(b) !1! 60:(a)$) {};
	\foreach \p in {a,b,v}
	{
		\fill (\p) circle (2pt);
	}
	\draw (a) -- (b) -- (v) -- (a);
	\draw (v) -- +(45:1) -- (v) -- +(65:1) -- (v) -- +(115:1) -- (v) -- +(+(135:1);
	\draw[dashed] (v) -- +(90:1);
	\begin{scope}[>=latex,very thick,densely dashed]
		\def\mytheta{30}
		\node[vtx] (tip) at ($(v) !0.5! ($(b) !0.5! (a)$)$) {};
		\draw[->] (v) -- (tip);
		\node at ($(v) !0.9! 10:(tip)$) {$\frac{4}{9}$};
	\end{scope}
	\node at ($(v) !1.5! (tip)$) {$f$};
\end{tikzpicture}\\
(R1a) &
(R2a) &
(R2b) \\
\begin{tikzpicture}[scale=\sixcyclerulescale,vtx/.style={shape=coordinate}]
	% tikz here
	\node[vtx] (a) at (0,0) {};
	\node[vtx] (b) at (180:3) {};
	\node[vtx] (c) at ($(b) !1! 90:(a)$) {};
	\node[vtx] (d) at ($(c) !1! 90:(b)$) {};
	\node[vtx] (e) at ($(c) !1/(2*cos(45))! 45:(d)$) {};
	\foreach \p in {a,b,c,d,e}
	{
		\fill (\p) circle (2pt);
	}
	\draw (a) -- (b) -- (c) -- (d) -- (e) -- (c) -- (d) -- (a);
	\draw (a) -- +(-30:1) -- (a) -- +(-60:1);
	\draw (b) -- +(-120:1) -- (b) -- +(210:1);
	\draw (c) -- +(150:1);
	\draw (d) -- +(30:1);
	\begin{scope}[>=latex,very thick,densely dashed]
		\def\mytheta{25}
		\draw[->] ($(b) !1/(2*cos(\mytheta))! -\mytheta:(a)$) -- ($(c) !1/(2*cos(\mytheta-10))! \mytheta-10:(d)$);
		\node at ($(b) !0.6! -15:(a)$) {$\frac{1}{8}$};
		\draw[->] ($(c) !1/(2*cos(\mytheta))! -\mytheta:(b)$) .. controls ($(c) !1/(2*cos(\mytheta))! \mytheta:(b)$) .. ($(c) !1/(4*cos(\mytheta))! \mytheta:(d)$);
		\node at ($(c) !0.35! -17:(b)$) {$\frac{1}{8}$};
		\draw[->] ($(a) !1/(2*cos(\mytheta))! -\mytheta:(d)$) .. controls ($(a) !1/(2*cos(\mytheta))! \mytheta:(d)$) .. ($(d) !1/(4*cos(\mytheta))! -\mytheta:(c)$);
		\node at ($(a) !0.7! -12:(d)$) {$\frac{1}{8}$};
	\end{scope}
	\node at ($(e) !0.3! -45:(d)$) {$f$};
	\node at ($(b) !0.8! 35:(a)$) {$g$};
	\node at ($(d) !0.4! 11:(c)$) {$e$};
	\node at ($(a) !0.3! -45:(d)$) {$h_1$};
	\node at ($(a) !0.4! 12:(d)$) {$e_1$};
	\node at ($(b) !0.3! -45:(a)$) {$h_2$};
	\node at ($(b) !0.37! 12:(a)$) {$e_2$};
	\node at ($(b) !0.3! 45:(c)$) {$h_3$};
	\node at ($(b) !0.4! -14:(c)$) {$e_3$};
\end{tikzpicture}
&
\begin{tikzpicture}[scale=\sixcyclerulescale,vtx/.style={shape=coordinate}]
	% tikz here
	\node[vtx] (a) at (0,0) {};
	\node[vtx] (b) at (180:3) {};
	\node[vtx] (c) at ($(b) !1! 90:(a)$) {};
	\node[vtx] (d) at ($(c) !1! 90:(b)$) {};
	\node[vtx] (e) at ($(c) !1/(2*cos(45))! 45:(d)$) {};
	\node[vtx] (f) at ($(b) !1/(2*cos(45))! 45:(c)$) {};
	\foreach \p in {a,b,c,d,e, f}
	{
		\fill (\p) circle (2pt);
	}
	\draw (a) -- (b) -- (c) -- (d) -- (e) -- (c) -- (d) -- (a);
	\draw (a) -- +(-30:1) -- (a) -- +(-60:1);
	\draw (b) -- +(-120:1) -- (b) -- +(210:1);
	\draw (b) -- (f) -- (c);
	\draw (c) -- +(150:1);
	\draw (d) -- +(30:1);
	\begin{scope}[>=latex,very thick,densely dashed]
		\def\mytheta{25}
		\draw[->] ($(b) !1/(2*cos(\mytheta))! -\mytheta:(a)$) -- ($(c) !1/(2*cos(\mytheta-10))! \mytheta-10:(d)$);
		\node at ($(b) !0.6! -15:(a)$) {$\frac{3}{16}$};
%		\draw[->] ($(c) !1/(2*cos(\mytheta))! -\mytheta:(b)$) .. controls ($(b) + (62:1.74)$) .. ($(b) !1/(4*cos(\mytheta))! -\mytheta:(a)$);
%		\node at ($(c) !0.35! -17:(b)$) {$\frac{3}{8}$};
		\draw[->] ($(a) !1/(2*cos(\mytheta))! -\mytheta:(d)$) .. controls ($(a) !1/(2*cos(\mytheta))! \mytheta:(d)$) .. ($(d) !1/(4*cos(\mytheta))! -\mytheta:(c)$);
		\node at ($(a) !0.7! -12:(d)$) {$\frac{3}{16}$};
	\end{scope}
	\node at ($(e) !0.3! -45:(d)$) {$f$};
	\node at ($(b) !0.8! 35:(a)$) {$g$};
	\node at ($(d) !0.4! 11:(c)$) {$e$};
	\node at ($(a) !0.3! -45:(d)$) {$h_1$};
	\node at ($(a) !0.4! 12:(d)$) {$e_1$};
	\node at ($(b) !0.3! -45:(a)$) {$h_2$};
	\node at ($(b) !0.37! 12:(a)$) {$e_2$};
	\node at ($(f) + (1,0)$) {$f_2$};
\end{tikzpicture}
&
\begin{tikzpicture}[scale=\sixcyclerulescale,vtx/.style={shape=coordinate}]
	% tikz here
	\node[vtx] (a) at (0,0) {};
	\node[vtx] (b) at (180:3) {};
	\node[vtx] (c) at ($(b) !1! 90:(a)$) {};
	\node[vtx] (d) at ($(c) !1! 90:(b)$) {};
	\node[vtx] (e) at ($(c) !1/(2*cos(45))! 45:(d)$) {};
	\node[vtx] (f) at ($(a) !1/(2*cos(45))! 45:(b)$) {};
	\foreach \p in {a,b,c,d,e, f}
	{
		\fill (\p) circle (2pt);
	}
	\draw (a) -- (b) -- (c) -- (d) -- (e) -- (c) -- (d) -- (a);
	\draw (a) -- +(-30:1) -- (a) -- +(-60:1);
	\draw (b) -- +(-120:1) -- (b) -- +(210:1);
	\draw (c) -- +(150:1);
	\draw (d) -- +(30:1);
	\draw (a) -- (f) -- (b);
	\begin{scope}[>=latex,very thick,densely dashed]
		\def\mytheta{25}
%		\draw[->] ($(b) !1/(2*cos(\mytheta))! -\mytheta:(a)$) -- ($(c) !1/(2*cos(\mytheta-10))! \mytheta-10:(d)$);
%		\draw[->] ($(a) !1/(2*cos(\mytheta))! +\mytheta:(b)$) .. controls ($(b) + (28:1.7)$) .. ($(b) !1/(4*cos(\mytheta))! +\mytheta:(c)$);
%		\node at ($(b) !0.6! -15:(a)$) {$\frac{3}{16}$};
		\draw[->] ($(c) !1/(2*cos(\mytheta))! -\mytheta:(b)$) .. controls ($(c) !1/(2*cos(\mytheta))! \mytheta:(b)$) .. ($(c) !1/(4*cos(\mytheta))! \mytheta:(d)$);
		\node at ($(c) !0.35! -17:(b)$) {$\frac{3}{16}$};
		\draw[->] ($(a) !1/(2*cos(\mytheta))! -\mytheta:(d)$) .. controls ($(a) !1/(2*cos(\mytheta))! \mytheta:(d)$) .. ($(d) !1/(4*cos(\mytheta))! -\mytheta:(c)$);
		\node at ($(a) !0.7! -12:(d)$) {$\frac{3}{16}$};
	\end{scope}
	\node at ($(e) !0.3! -45:(d)$) {$f$};
	\node at ($(b) !0.8! 35:(a)$) {$g$};
	\node at ($(d) !0.4! 11:(c)$) {$e$};
	\node at ($(a) !0.3! -45:(d)$) {$h_1$};
	\node at ($(a) !0.4! 12:(d)$) {$e_1$};
	\node at ($(b) !0.3! 45:(c)$) {$h_2$};
	\node at ($(b) !0.4! -14:(c)$) {$e_2$};
	\node at ($(f) + (0,1)$) {$f_2$};
\end{tikzpicture}
\\
(R1b)
&
(R1c), Case 1
&
(R1c), Case 2
\end{tabular}
\end{center}
\caption{Discharging rules (R1) and (R2) in the proof of Theorem~\ref{thm:cc7}.}\label{fig-rules7}
\end{figure}

\begin{itemize}
\item[(R0)] If $v$ is a precolored vertex and $f$ is an incident 3-face with negative charge, then $v$ sends charge $\frac{1}{2}$ to $f$.

\item[(R1)] If $f$ is a 3-face and $e$ is an incident edge, then let $g$ be the face adjacent to $f$ across $e$.

\begin{itemize}
\item[(R1a)] If $g$ is a $5^+$-face, then $f$ pulls charge $\frac{3}{8}$ from $g$ ``through'' the edge $e$.

\item[(R1b)] If $g$ is a 4-face and $f$ is the only 3-face adjacent to $g$, then let $e_1$, $e_2$, and $e_3$ be the other edges incident to $g$.
For each $i \in \{1,2,3\}$, let $h_i$ be the face adjacent to $g$ across $e_i$. For each $i \in \{1,2,3\}$, the face $f$ pulls charge $\frac{1}{8}$ from the face $h_i$ ``through'' the edges $e$ and $e_i$.

\item[(R1c)] If $g$ is a 4-face and $g$ is adjacent to two 3-faces $f_1$ and $f_2$ (say $f_1 = f$), then let $e_1$ and $e_2$ be the other edges incident to $g$, where the faces $h_1$ and $h_2$ sharing these edges are $6^+$-faces.
For each $i \in \{1,2\}$, the face $f$ pulls charge $\frac{3}{16}$ from the face $h_i$ ``through'' the edges $e$ and $e_i$.
\end{itemize}

\item[(R2)] Let $v$ be a $5^+$-vertex with $v \notin V(P)$ and let $f$ be an incident 3-face.

\begin{itemize}
\item[(R2a)] If $v$ is a $5$-vertex, then $v$ sends charge $\frac{1}{a}$ to $f$, when $a = \max\{ 3, t_3(v) \}$.

\item[(R2b)] If $v$ is a $6^+$-vertex, then $v$ sends charge $\frac{1}{2}$ to $f$.
\end{itemize}

\item[(R3)]  If $f$ is a 6-face with $\nu_2(f) < 0$ and $v$ is an incident $5^+$-vertex or an incident vertex in $V(P)$ with $\mu_0(v) > 0$, then $v$ sends charge $\frac{1}{4}$ to $f$.
\end{itemize}

We claim that $\mu_3(v) \geq 0$ for every vertex $v$ and $\nu_3(f) \geq 0$ for every face $f$.
Since the total charge sum was preserved during the discharging rules, this contradicts the negative charge sum from the initial charge values.

Note that 6-faces are not incident to 3-faces since $G$ does not contain a chorded 7-cycle.
Observe that a 6-face $f$ has $\nu_1(f) < 0$ if and only if all faces adjacent to $f$ are 4-faces, and each of those 4-faces has two adjacent 3-faces.

\begin{claim}\label{clm:precolorednonnegative}
Let $v$ be a vertex in $V(P)$.
Then $\mu_3(v) \geq 0$.
In addition, if $v$ is incident to a $6$-face $f$ with $\nu_1(f) < 0$, then $\mu_0(v) > 0$.
\end{claim}

\begin{proof}
By Claims~\ref{clm:precoloreddegree} and \ref{clm:bigsubgraph}, we have $\mu(v) = d(v) - 2 \geq 0$.
Note that if $\mu(v) \geq \frac{1}{2}t_3(v) + \frac{1}{4}t_6(v)$, then the final charge $\mu_3(v)$ is nonnegative.
Since $d(v) \geq t_3(v) + t_6(v)$, it suffices to show that $\mu_0(v) \geq \frac{1}{4}d(v) + \frac{1}{4}t_3(v)$.
%, but in some cases we will be more specific.

\begin{mycases}
\mycase{$P \cong P_3$.} Let $v_1$, $v_2$, and $v_3$ be the vertices in the 3-path $P$.
For $i\in \{1,2,3\}$, $\mu(v_i) = d(v_i) - 2$.
Since $P$ is not isomorphic to $K_3$, these vertices do not form a cycle, and the face to which all vertices are incident is not a 3-face.
Hence $t_3(v_i) \leq d(v_i) - 1$.
If $d(v_i) \geq 4$, then $\mu(v_i) = d(v_i) - 2 \geq \frac{1}{2}d(v_i ) > \frac{1}{4}d(v_i) + \frac{1}{4}t_3(v_i)$.

If $d(v_i) = 2$, then $\mu(v_i) = 0$.  
If $i = 2$, then $v_2$ is not incident to any 3-faces since $v_1$ and $v_3$ are not adjacent.
If $i\in\{1,3\}$ and $v_i$ is adjacent to a 3-face, then let $v_i'$ be the neighbor of $v_i$ not in $V(P)$.
Let $P'$ be the subgraph induced by $(V(P)  \cup \{v_i'\} )\setminus \{v_i\}$, which forms a copy of $P_3$ or $K_3$ in $G - v_i$.
For any color $c(v_i') \in L(v_i') \setminus\{c(v_i)\}$, there exists an $L$-coloring of $G - v_i$ as $(G-v_i,P',L,c)$ is not a counterexample; this coloring extends to an $L$-coloring of $G$.
Thus, $t_3(v_i) = 0$.
If $v_i$ is incident to a 6-face $f$ with $\nu_1(f) < 0$, then the other face incident to $v_i$ is a 4-face that is adjacent to two 3-faces. 
This results in a chorded 7-cycle, a contradiction; thus (R3) does not apply to $v_i$.
%Thus, for some neighbor $v_j$ of $v_i$ in $P$, $v_j$ is incident to $f$, $g$, and at least one 3-face; the 3-face is not adjacent to $f$, so $d(v_j) \geq 4$.
%Thus, $v_j$ is a positively-charged precolored vertex incident to $f$.

If $d(v_i) = 3$, Claim~\ref{clm:separating} asserts that $G$ has no separating 3-cycles, so then $v_i$ loses charge at most 1 in (R0).
If $v_i$ is incident to a 6-face $f$ with $\nu_1(f) < 0$, then the other two faces incident to $v_i$ are 4-faces and these 4-faces are each adjacent to two 3-faces. 
This creates a chorded 7-cycle, a contradiction, so (R3) does not apply to $v_i$ and $\mu_3(v_i) \geq 0$.
%so $t_3(v_i) = 0$.
%Thus, at most one of (R0) or (R3) applies to $v_i$ and the resulting charge is nonnegative.

\mycase{$P \cong K_3$.} Let $v_1$, $v_2$, and $v_3$ be the vertices in the 3-cycle $P$, so $\mu(v_i) = d(v_i) - 2$ for each $v_i$.
By Claim~\ref{clm:separating}, $G$ has no separating 3-cycle, so the three vertices are incident to a common 3-face $f$ with $\nu(f) = 0$.
Therefore, each vertex $v_i$ sends charge $\frac{1}{2}$ to at most $d(v_i) - 1$ incident 3-faces by (R0).
Recall that $d(v_i) \geq 3$ by Claim~\ref{clm:precoloreddegree}.
Suppose that $d(v_i) = 3$. 
If $t_3(v_i) > 1$, the subgraph of $G$ induced by the neighborhood of $v_i$ is isomorphic to $P_3$ or $K_3$, contradicting Claim~\ref{clm:precoloredtriangles}.
If $d(v_i) \geq 4$, then $\mu(v_i) = d(v_i) - 2 \geq \frac{1}{2}d(v_i ) \geq \frac{1}{4}d(v_i) + \frac{1}{4}t_3(v_i)$.
Therefore, $\mu_3(v_i) \geq 0$.
\end{mycases}

Thus, in all cases a precolored vertex $v$ has $\mu_3(v) \geq 0$.
\end{proof}

We will now show that all objects that start with nonnegative charge also end with nonnegative charge.

If $f$ is a 4-face, then (R1b) and (R1c) do not pull charge from $f$, since this would require $f$ to be adjacent to a 4-face $g$ that is adjacent to a 3-face $t$, but then $f$, $g$, and $t$ form a doubly-chorded 7-cycle.
Thus, $\nu_3(f) = 0$ for every 4-face $f$.

If $f$ is a 5-face, then since $G$ contains no chorded 7-cycles, $f$ is not adjacent to two 3-faces and $f$ is not adjacent to a 4-face.
Therefore, $f$ loses charge at most $\frac{3}{8}$ by (R1a), but loses no charge using (R1b), so $\nu_3(f) > 0$ for every $5$-face $f$.

If $f$ is a 6-face, then $f$ is not adjacent to a 3-face since $G$ contains no chorded 7-cycle.
Observe that by Claim~\ref{clm:2conn} the boundary of $f$ is a simple 6-cycle.
So if $f$ sends charge through an edge $e$ during (R1), it can send charge $\frac{1}{8}$ through $e$ by (R1b), or it can send charge $\frac{3}{8}$ through $e$ by (R1c).
The only way that this will result in a negative charge after (R1) and (R2) is for $f$ to send charge $\frac{3}{8}$ through each of its six edges by (R1c); this will cause $\nu_2(f) = 2 - 6\cdot \frac{3}{8} = -\frac{1}{4}$.
If $f$ has a precolored vertex $v$ on its boundary, then by Claim~\ref{clm:precolorednonnegative}, $v$ has positive charge after (R0); by (R3), $f$ receives charge at least $\frac{1}{4}$, resulting in $\nu_3(f) \geq 0$.
If $f$ has no incident precolored vertices, then since $G$ contains no \ref{cycle6}, some vertex $v$ on the boundary of $f$ is a $5^+$-vertex.
By (R3) $v$ sends charge $\frac{1}{4}$ to $f$ and hence $\nu_3(f) \geq 0$.
Observe the following claim concerning the structure about a vertex that loses charge by (R3).

\begin{claim}
Let $v$ be a $5^+$-vertex with the three incident faces $f_1$, $f_2$, and $f_3$, in cyclic order.
If $v$ sends charge to $f_2$ by (R3), then $f_1$ and $f_3$ are 4-faces and $f_2$ is a 6-face.
\end{claim}

If $f$ is a $7^+$-face, then by (R1) $f$ loses charge at most $\frac{3}{8}$ through each edge.
Thus,
\[
	\nu_3(f) \geq \ell(f) - 4 - \frac{3}{8}\ell(f) = \frac{5}{8}\ell(f) - 4  >  0.
\]
Therefore, $\nu_3(f) > 0$ for every $7^+$-face $f$.

Next, we will consider a vertex $v$ not in $V(P)$.

If $v$ is a 4-vertex, then $v$ does not lose charge by any rule, so the resulting charge is $0$.

If $v$ is a 5-vertex,  let $a = \max \{ 3, t_3(v)\}$ and $v$ loses charge $\frac{1}{a}t_3(v)$ to incident 3-faces by (R2a).
If (R3) does not apply to $v$, then $v$ sends charge at most 1 to  incident 3-faces and $\mu_3(v) \geq 0$.
If (R3) applies to $v$, then $v$ is incident to faces $f_1$, $f_2$, and $f_3$ where $f_1$ and $f_3$ are 4-faces and $f_2$ is a 6-face.
Since $d(v) = 5$ and $G$ has no chorded 7-cycle, the rule (R3) applies at most once.
If (R3) applies once, then $t_3(v) \leq 2$ and $v$ loses charge at most $\frac{2}{3}$ by (R2) and charge $\frac{1}{4}$ by (R3), so $\mu_3(v) \geq 0$.

If $v$ is a $6^+$-vertex, then let $k=t_3(v)$ and $\ell$ be the number of times (R3) applies to $v$.  
Notice that $k \leq \lfloor\frac{4}{5}d(v)\rfloor$ since $G$ avoids chorded 7-cycles.
Further, notice that $k + 2\ell \leq d(v)$, since each 6-face that gains charge from $v$ by (R3) is preceded by a 4-face in the cyclic order of faces around $v$.
By (R2b), $v$ can lose charge $\frac{1}{2}$ to each incident 3-face, and $v$ can lose charge at most $\frac{1}{4}$ to each incident 6-face by (R3).  
Then $v$ ends with charge 
\[
    \mu_3(v) \geq d(v) - 4 - \frac{1}{2}k - \frac{1}{4}\ell.
\]
If $d(v) = 6$, then observe $k + \ell \leq 4$ and hence $\mu_3(v) \geq 0$.
If $d(v) = d\geq 7$, then $d$, $k$, and $\ell$ satisfy the following linear program with dual on variables $a_1$, $a_2$, and $a_3$:
\[
	\begin{array}[h]{rrcrcrcl}
	\min & d &-& \frac{1}{2}k&-&\frac{1}{4}\ell&&\\
	\text{s.t.} & d && && &\geq& 7\\
	  	       & 4d &-&5k && &\geq&0\\
		       & d & - &k & - & 2\ell & \geq & 0\\
		       & d, && k,&&\ell&\geq &0
	\end{array}
	\qquad
	\begin{array}[h]{rrcrcrcl}
	\max & 7a_1\\
	\text{s.t.} & a_1 &+&5a_2 &+&a_3 &\leq& 1\\
	  	       &  &-&5a_2 &-&a_3 &\leq&-\frac{1}{2}\\
		       &  &  & & - & 2a_3 & \leq & -\frac{1}{4}\\
		       & a_1, && a_2,&&a_3&\geq &0
	\end{array}
\]
The dual-feasible solution $(a_1,a_2,a_3) = \left(\frac{23}{40}, \frac{1}{20}, \frac{1}{4}\right)$ demonstrates that $d - \frac{1}{2}k - \frac{1}{4}\ell \geq 7\cdot \frac{23}{40}> 4$, and thus $\mu_3(v) > 0$ for every $7^+$-vertex.

It remains to be shown that the clusters receive enough charge to become nonnegative.
Since $G$ contains no separating 3-cycle, $G$ does not contain the cluster (K5c) or the clusters (K6g)--(K6r).
Observe that there is no precolored vertex $v$ of degree at most three where all faces incident to $v$ have length three.
Finally, it is worth noting again that if $G$ contains a reducible configuration $(C,X,\ex)$, then there is a precolored vertex in the set $X$.

If a vertex $v$ is a $5^+$-vertex or $v \in V(P)$, we say $v$ is \emph{full}; if $v$ is a $6^+$-vertex or $v \in V(P)$, then $v$ is \emph{heavy}.
Note that a heavy vertex $v$ sends charge $\frac{1}{2}$ to each incident negatively-charged 3-face by (R0) or (R2b).
If $P \cong K_3$, we call $P$ the precolored face.

\begin{figure}[h]
\centering
\begin{tabular}[h]{VVV}
	\begin{tikzpicture}
	% tikz here
	\node[shape=coordinate] (a) at (0,0) {};
	\node[shape=coordinate] (b) at (0:1) {};
	\node[shape=coordinate] (c) at (60:1) {};

	\foreach \p in {a,b,c}
	{
		\fill (\p) circle (2pt);
	}

	\draw (a) -- (b) -- (c) -- (a);
	
	\node at (35:0.57) {$f$};
\end{tikzpicture}
	&
	\begin{tikzpicture}[scale=1]
	% tikz here
	\node[shape=coordinate] (a) at (180:1) {};
	\node[shape=coordinate, label=above:$v$] (b) at (120:1) {};
	\node[shape=coordinate] (c) at (0,0) {};
	\node[shape=coordinate] (d) at (60:1) {};

	\foreach \p in {a,b,c,d}
	{
		\fill (\p) circle (2pt);
	}

	\draw (a) -- (b) -- (c) -- (a);
	\draw (c) -- (d) -- (b);

	\node at (147:0.58) {$f_1$};
	\node at (90:0.57) {$f_2$};
\end{tikzpicture}
	&
	\begin{tikzpicture}[scale=1]
	% tikz here
	\node[shape=coordinate] (a) at (180:1) {};
	\node[shape=coordinate] (b) at (120:1) {};
	\node[shape=coordinate] (c) at (0,0) {};
	\node[shape=coordinate] (d) at (60:1) {};
	\node[shape=coordinate] (e) at (0:1) {};

	\foreach \p in {a,b,c,d,e}
	{
		\fill (\p) circle (2pt);
	}

	\draw (a) -- (b) -- (c) -- (d) -- (e) -- (c) -- (a);
	\draw (d) -- (b);

	\node at (150:0.57) {$f_1$};
	\node at (90:0.57) {$f_2$};
	\node at (30:0.57) {$f_3$};

\end{tikzpicture}
	\\
	(K3) & (K4) & (K5a)
\end{tabular}
\caption{Clusters (K3), (K4), and (K5a)}
\end{figure}
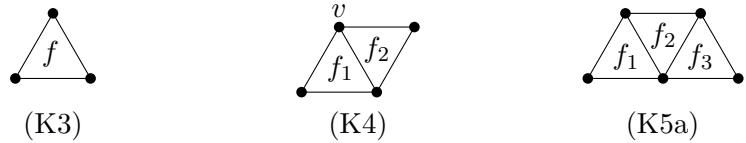
  
  \vspace{-2em}
\begin{mycases}
\mycase{(K3)} 
Let $f$ be the isolated 3-face in (K3).  
If $f$ is the precolored face, then $\nu_3(f) = \nu(f) = 0$. 
Otherwise, the initial charge on $f$ is $-1.$ 
By (R1), $f$ receives charge $\frac{9}{8}$ through its boundary edges, resulting in a nonnegative final charge.

\mycase{(K4)} Let $f_1$ and $f_2$ be 3-faces in a diamond cluster (K4). 
First, suppose without loss of generality that $f_1$ is the precolored face. 
The initial charge of the cluster is $-1$. 
Then $f_2$ receives charge $1$ by (R0) and charge $2 \cdot \frac{3}{8}$ by (R1), resulting in a positive final charge. 
Otherwise, the initial charge on the cluster is $-2.$ 
By (R1), $f_1$ and $f_2$ receive charge $\frac{3}{8}$ through each of the two edges on the boundary of the cluster, resulting in charge $-\frac{1}{2}.$ 
If the cluster contains a precolored vertex $u$, then it receives charge $\frac{1}{2}$ by (R0). 
Otherwise, since G contains no \ref{diamond1}, there is a $5^+$ -vertex $v$ incident to both $f_1$ and $f_2.$ 
By (R2), this vertex sends charge at least $\frac{1}{3}$ to each of the faces, resulting in a nonnegative final charge.

% If $v$ is a $6^+$-vertex, then by (R2b), $f_1$ and $f_2$ each receive charge $\frac{1}{2}$ from $v$, and the resulting charge on the diamond cluster is nonnegative. If $v$ is a 5-vertex, then by (R2a), $v$ sends charge at least $\frac{1}{4}$ to $f_1$ and $f_2,$ and the resulting charge on the diamond cluster is nonnegative.

\mycase{(K5a)} Let $f_1$, $f_2$, and $f_3$ be 3-faces in a 3-fan cluster (K5a), where $f_2$ is adjacent to both $f_1$ and $f_3$. 
Suppose that the cluster contains a precolored face, so the initial charge on the cluster is $-2$. 
If $f_2$ is precolored, then the cluster receives charge $4 \cdot \frac{1}{2}$ by (R0); if $f_1$ or $f_3$ is precolored, then the cluster receives charge $3\cdot \frac{1}{2}$ by (R0) and charge $3 \cdot\frac{3}{8}$ by (R1). 
In either case, the final charge is nonnegative. 

If $P\not\cong K_3$ or the cluster does not contain the precolored face, then the initial charge on the cluster is $-3.$  
By (R1), the cluster receives charge $5\cdot \frac{3}{8},$ resulting in charge $-\frac{9}{8}$. 
Note that the faces $f_1$ and $f_2$ form a diamond and the faces $f_2$ and $f_3$ form a diamond.  
Since $G$ contains no \ref{diamond1}, there exists a full vertex $v$ incident to both $f_1$ and $f_2$. 
Similarly, there exists a full vertex $u$ incident to $f_2$ and $f_3.$ 
If $u \neq v,$ then by (R0) or (R2), $v$ sends charge at least $\frac{1}{3}$ to each of $f_1$ and $f_2$ and $u$ sends charge at least $\frac{1}{3}$ to each of $f_2$ and $f_3$, resulting in nonnegative charge on the cluster. 
If $u = v$ and $v$ is a heavy vertex, then $v$ sends charge $\frac{1}{2}$ to each face $f_1$, $f_2$, and $f_3$, resulting in nonnegative charge on the cluster. 
Otherwise, suppose that $u = v \notin V(P)$ and $v$ is a 5-vertex. 
Since $G$ contains no \ref{diamond2}, there exists another full vertex $w$ that is incident to at least one of $f_1$ and $f_2$. 
By (R2a), $v$ sends charge $\frac{1}{3}$ to $f_1$, $f_2,$ and $f_3$, and by (R0) or (R2), $w$ sends charge at least $\frac{1}{3}$ to one of $f_1$ and $f_2,$ resulting in nonnegative charge on the cluster.

\begin{figure}[h]
\centering
\begin{tabular}[h]{VVV}
	\begin{tikzpicture}[scale=0.66]
	% tikz here
	\node[shape=coordinate, label=above:$v$] (a) at (0,0) {};	
	\node[shape=coordinate, label=above:$u_2$] (b) at (45:1.5) {};
	\node[shape=coordinate, label=above:$u_1$] (c) at (135:1.5) {};
	\node[shape=coordinate, label=below:$u_4$] (d) at (-135:1.5) {};
	\node[shape=coordinate, label=below:$u_3$] (e) at (-45:1.5) {};

	\foreach \p in {a,b,c,d,e}
	{
		\fill (\p) circle (2pt);
	}

	\draw (a) -- (b) -- (c) -- (d) -- (e) -- (b);
	\draw (c) -- (a) -- (d);
	\draw (a) -- (e);

	\node at (90:0.75) {$f_1$};
	\node at (0:0.75) {$f_2$};
	\node at (-90:0.7) {$f_3$};
	\node at (180:0.75) {$f_4$};
	
\end{tikzpicture}
	&
	\begin{tikzpicture}[scale=1]
	% tikz here
	\node[shape=coordinate] (a) at (180:1) {};	
	\node[shape=coordinate] (b) at (120:1) {};
	\node[shape=coordinate] (c) at (0,0) {};
	\node[shape=coordinate] (d) at (60:1) {};
	\node[shape=coordinate] (e) at (0:1) {};
	\node[shape=coordinate] (f) at (30:1.732) {};

	\foreach \p in {a,b,c,d,e,f}
	{
		\fill (\p) circle (2pt);
	}

	\draw (a) -- (b) -- (c) -- (d) -- (e) -- (f) -- (d) -- (b);
	\draw (a) -- (c) -- (e);

	\node at (150:0.57) {$f_1$};
	\node at (90:0.57) {$f_2$};
	\node at (30:0.57) {$f_3$};
	\node at (30:1.15) {$f_4$};

\end{tikzpicture} 
	&
	\begin{tikzpicture}[scale=1]
	% tikz here
	\node[shape=coordinate, label=below right:$v$] (a) at (0,0) {};
	\node[shape=coordinate, label=below:$u_5$] (b) at (0:1) {};
	\node[shape=coordinate, label=right:$u_4$] (c) at (60:1) {};
	\node[shape=coordinate, label=left:$u_3$] (d) at (120:1) {};
	\node[shape=coordinate, label=left:$u_2$] (e) at (180:1) {};
	\node[shape=coordinate, label=right:$u_1$] (f) at (240:1) {};

	\foreach \p in {a,b,c,d,e,f}
	{
		\fill (\p) circle (2pt);
	}

	\draw (a) -- (b) -- (c) -- (d) -- (e) -- (f) -- (a) -- (c);
	\draw (d) -- (a) -- (e);

	\node at (213:0.57) {$f_1$};
	\node at (150:0.57) {$f_2$};
	\node at (90:0.57) {$f_3$};
	\node at (30:0.57) {$f_4$};

\end{tikzpicture} 
	\\
	(K5b) & (K6a) & (K6b)
\end{tabular}
\caption{Clusters (K5b), (K6a), and (K6b)}
\end{figure}

\mycase{(K5b)} Let $f_1$, $f_2$, $f_3$, and $f_4$ be 3-faces in a 4-wheel (K5b).
If the cluster contains a precolored face, then the initial charge on the cluster is $-3$; the cluster receives charge $5 \cdot \frac{1}{2}$ by (R0) and charge $3 \cdot \frac{3}{8}$ by (R1), resulting in a positive final charge.
Otherwise, the initial charge on this cluster is $-4$. 
By (R1), the cluster receives charge $4\cdot \frac{3}{8}$, resulting in charge $-\frac{5}{2}$.
%There are four edges on the boundary of this cluster, and each is incident with a $7^+$-face since the 4-wheel is a maximal cluster in a chorded 7-cycle-free graph.
%So by (R1a), the cluster receives charge $\frac{3}{2}$, resulting in charge $-\frac{5}{2}$ after (R1).
Let $v$ be the 4-vertex incident to all four 3-faces. 
Let $u_1$, $u_2$, $u_3$, and $u_4$ be the vertices adjacent to $v$, ordered cyclically such that $vu_iu_{i+1}$ is the boundary of the 3-face $f_i$ for $i\in \{1,2,3\}$ and $vu_4u_1$ is the boundary of $f_4$.  
Since the cluster does not contain the precolored face, $v$ is not a precolored vertex.
%Note that this is because all precolored vertices share a common face.
%If $v$ is a precolored vertex, then it contributes charge $4 \cdot \frac{1}{2}$ to the cluster by (R0). Also by Claim~\ref{clm:bigsubgraph}, at least one neighbor $u_i$ of $v$ is precolored and contributes charge $2\cdot \frac{1}{2}$ to the cluster by (R0), resulting in nonnegative charge.
%Now suppose $v$ is not a precolored vertex.
Since $G$ contains no \ref{diamond1}, each $u_i$ is a full vertex.
%If $u_i$ is not a heavy vertex, then since $G$ contains no \ref{diamond1}, $u_i$ is a 5-vertex.
When $u_i$ is a 5-vertex, it is incident to two $7^+$-faces, so $u_i$ sends charge $\frac{1}{3}$ to each incident 3-face by (R2).
Thus, each $u_i$ sends charge at least $2\cdot \frac{1}{3}$ to the cluster by (R0) or (R2), resulting in a nonnegative final charge.

\mycase{(K6a)} Let $f_1$, $f_2$, $f_3$, and $f_4$ be 3-faces in a 4-strip cluster (K6a).
If the cluster contains the precolored face, then the initial charge on the cluster is $-3$.
If $f_1$ or $f_4$ is precolored, then the cluster receives charge $3 \cdot \frac{1}{2}$ by (R0) and charge $4 \cdot \frac{3}{8}$ by (R1);
if $f_2$ or $f_3$ is precolored, then the cluster receives charge $5\cdot \frac{1}{2}$ by (R0) and charge $5\cdot \frac{3}{8}$ by (R1). 
In either case, the resulting final charge is nonnegative.  
If the cluster does not contain the precolored face, then the initial charge on this cluster is $-4$. 
By (R1), the cluster receives charge $6 \cdot \frac{3}{8}$, resulting in charge $-\frac{7}{4}$.
Note that for $i \in \{ 1, 2, 3\}$, the faces $f_i$ and $f_{i+1}$ form a diamond.  
Since $G$ contains no \ref{diamond1}, there exists a full vertex $v$ incident to both $f_i$ and $f_{i+1}$.
Let $u_1$ be a full vertex incident to $f_2$ and $f_3$. 
Without loss of generality, $u_1$ is not incident to $f_4$, so there is a full vertex $u_2$ incident to $f_1$ and $f_2$.
If $u_1$ is a heavy vertex, the cluster receives charge $3 \cdot \frac{1}{2}$ from $u_1$ by (R0) or (R2b), and charge at least $2\cdot \frac{1}{3}$ from $u_2$ by (R0) or (R2), resulting in a positive final charge. 
Otherwise, $u_1$ is a $5$-vertex, so $u_1$ sends charge $3 \cdot \frac{1}{3}$ by (R2a), resulting in charge $-\frac{3}{4}$. 
If $u_2$ is incident to $f_3$, then $u_2$ sends charge at least $3 \cdot \frac{1}{3}$ by (R0) or (R2), resulting in a positive final charge. 
Otherwise, $u_2$ is incident with $f_1$ and $f_2$ but not $f_3$. 
If $u_2$ is a large vertex, it sends charge $2 \cdot \frac{1}{2}$ by (R0) or (R2b).
Otherwise, since $G$ contains neither a \ref{diamond1} or a \ref{diamond2}, there is a third full vertex $u_3$.  
The cluster receives charge $2 \cdot \frac{1}{3}$ from $u_2$ by (R2a) and charge at least $\frac{1}{3}$ from $u_3$ by (R0) or (R2). 
In each case, the resulting final charge is nonnegative.

\mycase{(K6b)} Let $f_1$, $f_2$, $f_3$, and $f_4$ be 3-faces in a 4-fan cluster (K6b).
Let $v$ be the center of the fan, with neighbors $u_1$, $u_2$, $u_3$, $u_4$, and $u_5$ where for $i\in \{1,2,3\}$, $f_i$ and $f_{i+1}$ are adjacent on the edge $vu_{i+1}$.
If the cluster contains the precolored face, then the initial charge on the cluster is $-3$.
If $f_1$ or $f_4$ is precolored, then the cluster receives charge $4 \cdot \frac{1}{2}$ by (R0) and charge $4 \cdot \frac{3}{8}$ by (R1);
if $f_2$ or $f_3$ is precolored, then the cluster receives charge $5\cdot \frac{1}{2}$ by (R0) and charge $5\cdot \frac{3}{8}$ by (R1). 
In either case, the resulting final charge is positive.

If the cluster does not contain the precolored face, then the initial charge on this cluster is $-4$.
By (R1), the cluster receives charge $6 \cdot \frac{3}{8}$, resulting in charge $-\frac{7}{4}$.
If $v$ is a heavy vertex, then by (R0) or (R2b) $v$ sends charge $4\cdot \frac{1}{2}$ to the cluster, resulting in positive charge.
Otherwise, $v \notin V(P)$ and $v$ is a 5-vertex, so $v$ sends charge $1$ to the cluster by (R2a), resulting in charge $-\frac{3}{4}$.
If there is a heavy vertex in $\{u_2, u_3, u_4\}$, then that vertex contributes charge $2 \cdot \frac{1}{2}$ to the cluster, resulting in a positive charge.
If there is no heavy vertex in $\{u_2, u_3, u_4\}$, then there is at least one $5$-vertex in $\{u_2, u_3, u_4\}$ since $G$ contains no \ref{diamond2}.  
If there are multiple $5$-vertices in $\{u_2,u_3,u_4\}$, then each sends charge $2 \cdot \frac{1}{3}$ to the cluster by (R2a), resulting in positive charge.
If there is only $5$-vertex $w$ among $u_2, u_3$, and $u_4$, then there is a full vertex $z \in \{u_1,u_5\}$ since $G$ does not contain \ref{diamond2} or \ref{4fan}; the cluster receives charge $2\cdot \frac{1}{3}$ from $w$ by (R2a) and at least $\frac{1}{3}$ from $z$ by (R0) or (R2), resulting in positive final charge.

\begin{figure}[h]
\centering
\begin{tabular}[h]{VV}
	\begin{tikzpicture}
	% tikz here
	\node[shape=coordinate, label=below:$u_5$] (a) at (180:1) {};	
	\node[shape=coordinate, label=left:$u_6$] (b) at (120:1) {};
	\node[shape=coordinate, label=below:$u_4$] (c) at (0,0) {};
	\node[shape=coordinate, label=right:$u_2$] (d) at (60:1) {};
	\node[shape=coordinate, label=below:$u_3$] (e) at (0:1) {};
	\node[shape=coordinate, label=right:$u_1$] (f) at (90:1.732) {};

	\foreach \p in {a,b,c,d,e,f}
	{
		\fill (\p) circle (2pt);
	}

	\draw (c) -- (a) -- (b) -- (c) -- (d) -- (b) -- (f) --(d) -- (e) -- (c);

	\node at (90:1.2) {$f_1$};
	\node at (30:0.57) {$f_2$};
	\node at (150:0.57) {$f_3$};
	\node at (90:0.57) {$f_4$};

\end{tikzpicture}
	&
	\begin{tikzpicture}[scale=0.75]
	% tikz here
	\node[shape=coordinate, label=above:$v$] (a) at (0,0) {};	
	\node[shape=coordinate, label=above:$u_3$] (b) at (45:1.5) {};
	\node[shape=coordinate, label=above:$u_2$] (c) at (135:1.5) {};
	\node[shape=coordinate, label=below:$u_1$] (d) at (-135:1.5) {};
	\node[shape=coordinate, label=below:$u_4$] (e) at (-45:1.5) {};
	\node[shape=coordinate, label=above:$w$] (f) at (0:2.12132) {};

	\foreach \p in {a,b,c,d,e,f}
	{
		\fill (\p) circle (2pt);
	}

	\draw (a) -- (b) -- (c) -- (d) -- (e) -- (b);
	\draw (c) -- (a) -- (d);
	\draw (a) -- (e);
	\draw (b) -- (f) -- (e);

	\node at (-90:0.6) {$f_1$};
	\node at (180:0.7) {$f_2$};
	\node at (90:0.65) {$f_3$};
	\node at (0:0.7) {$f_4$};
	\node at (0:1.45) {$g$};

\end{tikzpicture}
	\\
	(K6c) & (K6d) 
\end{tabular}
\caption{Clusters (K6c) and (K6d).}
\end{figure}
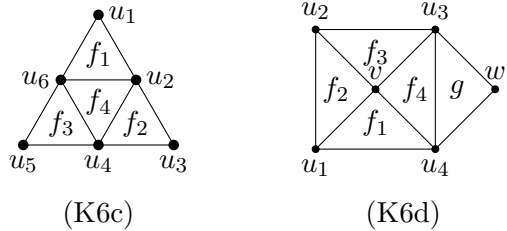

\mycase{(K6c)} Let $f_1$, $f_2$, $f_3$, and $f_4$ be the 3-faces of this cluster (K6c) where $f_4$ is adjacent to each $f_i$ for $i \in \{1,2,3\}$.
If the cluster contains the precolored face, then the initial charge on the cluster is $-3$.
If one of $f_1, f_2$ or $f_3$ is precolored, the cluster receives charge $4 \cdot \frac{1}{2}$ by (R0) and charge $4 \cdot \frac{3}{8}$ by (R1). 
If $f_4$ is precolored, then the cluster receives charge $6\cdot \frac{1}{2}$ by (R0). 
In either case, the resulting final charge is nonnegative.

If the cluster does not contain the precolored face, then the initial charge on the cluster is $-4$.
By (R1), the cluster receives charge $6 \cdot \frac{3}{8}$, resulting in charge $-\frac{7}{4}$.
Let $u_1$, $u_2$, $u_3$, $u_4$, $u_5$, and $u_6$ be the vertices on the boundary of the cluster ordered such that $u_2, u_4, u_6$ are the vertices incident to $f_1$ and $f_2$, $f_2$ and $f_3$, and $f_3$ and $f_1$, respectively.
Since $G$ contains no \ref{diamond1}, there are at least two full vertices in $\{u_2, u_4, u_6\}$.
By (R0) or (R2), these vertices each send charge at least $1$ to the cluster, resulting in a positive total charge.

\mycase{(K6d)} Let $f_1$, $f_2$, $f_3$, and $f_4$ be cyclically-ordered 3-faces in a 4-wheel with center vertex $v$ where $f_i$ and $f_{i+1}$ share a common edge for $i\in \{1,2,3,4\}$, where indices are taken modulo 4; let $g$ be a 3-face adjacent to $f_4$ but not incident to $v$, completing our cluster (K6d).
If the cluster contains the precolored face, then the initial charge on the cluster is $-4$.
If $f_1$ or $f_3$ is precolored,  then the cluster receives charge $6 \cdot \frac{1}{2}$ by (R0) and charge $4 \cdot \frac{3}{8}$ by (R1).
If $f_2$ is precolored, then the cluster receives charge $5\cdot \frac{1}{2}$ by (R0) and charge $4 \cdot \frac{3}{8}$ by (R1). 
If $f_4$ is precolored, then  the cluster receives charge $7\cdot \frac{1}{2}$ by (R0) and charge $5 \cdot \frac{3}{8}$ by (R1). 
In each of the above cases, the final charge is nonnegative.
If $g$ is precolored,  then the cluster receives charge $4\cdot \frac{1}{2}$ by (R0) and charge $3 \cdot \frac{3}{8}$ by (R1), resulting in charge $-\frac{7}{8}$. 
Let $N(v) = \{ u_1, u_2, u_3, u_4 \}$ where $u_i$ is incident to $f_i$ and $f_{i+1}$ for all $i\in \{1,2,3,4\}$. 
Since $G$ does not contain \ref{diamond1}, $u_1$ and $u_2$ are full vertices. 
Each of $u_1$ and $u_2$ sends charge at least $2 \cdot \frac{1}{3}$ to the cluster by (R2), resulting in nonnegative charge.

If the cluster does not contain the precolored face, then the initial charge on this cluster is $-5$ and $v \notin V(P)$.
By (R1), the cluster receives charge $5 \cdot \frac{3}{8}$, resulting in charge $-\frac{25}{8}$.
Since $G$ does not contain \ref{diamond1}, $u_1$, $u_2$, $u_3$, and $u_4$ are full vertices.
By (R0) or (R2), the cluster receives charge at least $2 \cdot \frac{1}{3}$ from each of $u_1$ and $u_2$ and charge at least $3 \cdot \frac{1}{3}$ from each of $u_3$ and $u_4$, resulting in a positive final charge.

\begin{figure}[h]
\centering
\begin{tabular}[h]{VV}
	\begin{tikzpicture}[scale=0.75]
	% tikz here
	\node[shape=coordinate] (a) at (0,0) {};
	\node[shape=coordinate, label=above right:$v$] (a1) at (-0.1,-0.05) {};
	\node[shape=coordinate, label=above:$u_1$] (b) at (18:1.5) {};
	\node[shape=coordinate, label=above:$u_5$] (c) at (90:1.5) {};
	\node[shape=coordinate, label=above:$u_4$] (d) at (162:1.5) {};
	\node[shape=coordinate, label=left:$u_3$] (e) at (234:1.5) {};
	\node[shape=coordinate, label=right:$u_2$] (f) at (306:1.5) {};

	\foreach \p in {a,b,c,d,e,f}
	{
		\fill (\p) circle (2pt);
	}

	\draw (a) -- (b) -- (c) -- (d) -- (e) -- (f) -- (b);
	\draw (b) -- (a) -- (c);
	\draw (d) -- (a) -- (e);
	\draw (a) -- (f);

	\node at (54:0.85) {$f_1$};
	\node at (342:0.85) {$f_2$};
	\node at (270:0.85) {$f_3$};
	\node at (198:0.85) {$f_4$};
	\node at (126:0.85) {$f_5$};

\end{tikzpicture} 
	&
	\begin{tikzpicture}[scale=1]
	% tikz here
	\node[shape=coordinate, label=below left:$u_1$] (a) at (0,0) {};
	\node[shape=coordinate] (b) at (180:1.25) {};
	\node[shape=coordinate, label=above:$z$] (c) at (60:1) {};
	\node[shape=coordinate, label=below:$w$] (d) at (-60:1) {};
	\node[shape=coordinate, label=below right:$u_2$] (e) at (0:1) {};
	\node[shape=coordinate] (f) at (0:2.25) {};

	\foreach \p in {a,b,c,d,e,f}
	{
		\fill (\p) circle (2pt);
	}

	\draw (b) -- (a) -- (c) -- (e) -- (d) -- (a) -- (e) -- (f);

	\node at (33:0.6) {$f_1$};
	\node at (-32:0.57) {$f_2$};

	\draw (b) .. controls +(90:0.5) and +(180:0.5) .. (c);
	\draw (c) .. controls +(0:0.5) and +(90:0.5) .. (f);
	\draw (f) .. controls +(-90:0.5) and +(0:0.5) .. (d);
	\draw (d) .. controls +(180:0.5) and +(-90:0.5) .. (b);
\end{tikzpicture} 
	\\
	(K6e) & (K6f) 
\end{tabular}
\caption{Clusters (K6e) and (K6f).}
\end{figure}
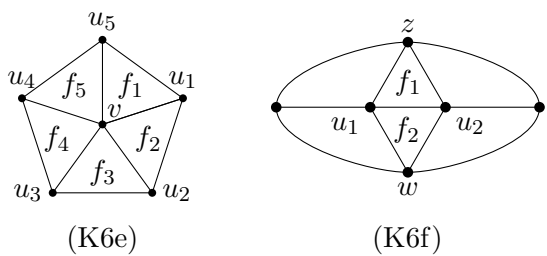

\mycase{(K6e)} Let $f_1$, $f_2$, $f_3$, $f_4$, and $f_5$ be the cyclically-ordered 3-faces in a 5-wheel with center vertex $v$ where $f_i$ and $f_{i+1}$ share a common edge for $i \in \{1,2,3,4,5\}$, where indices are taken modulo 5.
Let $N(v) = \{u_1, u_2, u_3, u_4, u_5\}$ where $u_i$ is incident to $f_i$ and $f_{i+1}$ for $i\in \{1,2,3,4,5\}$.
If the cluster contains the precolored face, then the initial charge on the cluster is $-4$. 
The cluster receives charge $6\cdot \frac{1}{2}$ by (R0) and charge $4 \cdot \frac{3}{8}$ by (R1), resulting in a positive final charge.

If the cluster does not contain the precolored face, then the initial charge is $-5$ and $v \notin V(P)$.
By (R1), the cluster receives charge $5 \cdot \frac{3}{8}$, and by (R2), the cluster receives charge $1$ from $v$, resulting in charge $-\frac{17}{8}$.
Since $G$ does not contain \ref{diamond2} or \ref{K6hconfig1}, there are at least three full vertices in $N(v)$. 
If $N(v)$ contains at least three heavy vertices, then the cluster receives charge at least $6 \cdot \frac{1}{2}$ by (R0) or (R2b), resulting in a positive final charge.
If $N(v)$ contains exactly two heavy vertices, then the cluster receives charge $4 \cdot \frac{1}{2}$ by (R0) or (R2b) and charge $2 \cdot \frac{1}{3}$ from a full vertex by (R2a), resulting in positive charge.
If $N(v)$ contains exactly one heavy vertex, then the cluster receives charge $2 \cdot \frac{1}{2}$ by (R0) or (R2b) and charge $2\cdot \frac{1}{3}$ from each of two full vertices by (R2a), resulting in positive final charge. 

If $N(v)$ contains no heavy vertices, then there are at least three full vertices in $N(v)$.
Since $G$ does not contain \ref{diamond2}, there are at least two nonadjacent 5-vertices in $N(v)$. 
Further, since $G$ does not contain \ref{K6hconfig1}, \ref{K6hconfig2}, or \ref{K6hconfig3}, there are at least four 5-vertices in $N(v)$. 
The cluster receives charge $2 \cdot \frac{1}{3}$ from each of these vertices by (R2a), resulting in a positive final charge.

% This does not have a separating 3-cycle, but DOES contain a reducible configuration!
\mycase{(K6f)} Let $f_1$ and $f_2$ be the interior 3-faces in the two overlapping 4-wheels that make up the cluster (K6f). 
Let $u_1$ and $u_2$ be the shared vertices of $f_1$ and $f_2$ and let $z$ and $w$ be the vertices incident with $f_1$ and $f_2$, respectively, that have not yet been labeled. 
Since $G$ contains no \ref{diamond1}, at least one of $u_1$ and $u_2$ is in $V(P)$. 
Then since all the precolored vertices lie on a common face, the cluster contains the precolored face, so the initial charge is $-5$. 
If $f_1$ or $f_2$ is precolored, then the cluster receives charge $8 \cdot \frac{1}{2}$ by (R0) and charge $4 \cdot \frac{3}{8}$ by (R1), resulting in a positive final charge. 
If one of the other 3-faces is precolored, then the cluster receives charge $6\cdot \frac{1}{2}$ by (R0) and charge $3\cdot \frac{3}{8}$ by (R1), resulting in charge $-\frac{7}{8}$. 
Since $G$ contains no \ref{diamond1}, one of $w$ and $z$ is a non-precolored $5^+$-vertex. 
This vertex sends charge at least $3\cdot \frac{1}{3}$ to the cluster by (R2), resulting in a positive final charge.
\end{mycases}

We have verified that the total charge after discharging is nonnegative, contradicting the negative initial charge sum.
Thus, a minimal counterexample does not exist and every planar graph with no chorded 7-cycle is $(4,2)$-choosable.
\end{proof}

%%%%%%%%%%%%%%%%%%%%%%%%%%%%%%%%%%%%%%%%%%
%%%%%%%%%%%%%%%%%%%%%%%%%%%%%%%%%%%%%%%%%%
%%%%%%%%%%%%%%%%%%%%%%%%%%%%%%%%%%%%%%%%%%
%%%%%%%%%%%%%%%%%%%%%%%%%%%%%%%%%%%%%%%%%%
%%%%%%%%%%%%%%%%%%%%%%%%%%%%%%%%%%%%%%%%%%
%%%%%%%%%%%%%%%%%%%%%%%%%%%%%%%%%%%%%%%%%%
%%%%%%%%%%%%%%%%%%%%%%%%%%%%%%%%%%%%%%%%%%
%%%%%%%%%%%%%%%%%%%%%%%%%%%%%%%%%%%%%%%%%%

\section{Conclusion and Future Work}

We proved that, for each $k \in \{5,6,7\}$, planar graphs with no chorded $k$-cycles are $(4,2)$-choosable. 
Our methods for proving reducible configurations created several large classes of reducible configurations, such as templates; naturally, there are many more reducible configurations than the ones we explicitly used.
One could likely prove that if $G$ is a planar graph with no chorded 4-cycle and no doubly-chorded 7-cycle, then $G$ is (4,2)-choosable using methods similar to those in this paper. 
We were unable to extend these results to prove Conjecture \ref{conj:42choosable}, that all planar graphs are $(4,2)$-choosable. 
%Also Conjecture \ref{conj:31choosable}, if $G$ is a planar graph, then $G$ is $(3,1)$-choosable, still remains open. 

\section*{Acknowledgments}
We thank Ryan R. Martin, Alex Nowak, Alex Schulte, and Shanise Walker for participation in the early stages of the project.

%%%%%%%%%%%%%%%%%%%%%%%%%%%%%%%%%%%%%%%%%%
{\small
\bibliographystyle{abbrv}
\bibliography{choosability.bib}
}

%\clearpage
\small
\appendix
\section{Large Reducible Configurations}\label{appx:reducible}

In the proof of Theorem~\ref{thm:5cycles}, we demonstrated that no minimal counterexample exists by showing that there exists a reducible configuration $(C,X,\ex)$ where $G$ contains a copy of $C[X]$ as an induced subgraph (and also the copy agrees with the external degrees).
In this appendix, we provide the details that clarify this assumption.
By Lemma~\ref{lma:iteration}, we can relax the condition that $C[X]$ is an induced subgraph.
We will demonstrate that the configurations that appear after some vertices in $X$ are merged (while also preserving the face lengths, vertex degrees, and lack of chorded 5-cycle) result in reducible configurations.

Let $(C,X, \ex)$ be a reducible configuration and let $\{x_1,x_1'\},\dots,\{x_t,x_t'\}$ be a list of vertex pairs in $X$.
For these configurations, we may identify some 3-cycles and 5-cycles that are required to be 5-faces (in the context of the proof of Theorem~\ref{thm:5cycles}).
The resulting configuration $(C',X',\ex)$ where $C'$ and $X'$ are modified from $C$ and $X$ by merging $x_i$ with $x_i'$ and removing any multiedges or loops that result.
We say a list $\{x_1,x_1'\},\dots,\{x_t,x_t'\}$ is \emph{valid} for $(C,X,\ex)$ if the resulting configuration $(C',X',\ex)$ may appear in a planar graph of minimum degree at least four containing no chorded 5-cycle.
There are three situations that can occur when we perform this action.

\noindent\textbf{Pairs too close:} If some pair $\{x_i, x_i'\}$ have $d(x_i,x_i') \leq 2$, then either we create a loop or a multiedge when merging $x_i$ and $x_i'$. This will reduce the degree of the resulting vertex, in addition to possibly shortening known 3- and 5-cycles. Since distances only decrease as vertices are merged, a pair failing this property will not appear in any valid list of pairs.

\noindent\textbf{Pairs creating chord:} If merging $x_i$ and $x_i'$ creates a chorded 5-cycle, then this configuration would not appear in the minimal counterexample from Theorem~\ref{thm:5cycles}. Since distances only decrease as vertices are merged, a pair failing this property will not appear in any valid list of pairs.

\noindent\textbf{Reducible pairs:} If merging $x_i$ and $x_i'$ does not fit in the above two cases, then we will demonstrate that the resulting configuration is reducible. Even if merging one pair of vertices creates a reducible configuration, we need to check all possible lists of pairs that contain that pair.

After considering all pairs that could be identified, observe that in each case there is no set of three or more vertices where every pair can be identified.

In the following tables, we list one of the configurations \ref{d2}--\ref{lastconf}, label the vertices, and list all pairs of vertices into the three categories above.
In the case of reducible pairs, we present the contracted graph.
Most of these contracted graphs contain a copy of \ref{cycle4}, \ref{cycle6}, \ref{d2}, \ref{3paths}, or \ref{3pathsB}.
The only exceptions are the contracted graphs derived from \ref{bigneedy}, but each of these configurations has an Alon-Tarsi orientation and hence is reducible.

\input{figures/appendix.tex}

%\section{Visual Reducibility Proofs}\label{appx:visual}
%
%\input{figures/visualproofs.tex}
%
%\clearpage
%\input{figures/bigconfiglist.tex}

\end{document}